\documentclass[12pt]{amsart}
\usepackage{amsfonts,amssymb,amscd,amstext,mathrsfs}
\usepackage[utf8]{inputenc}
\usepackage{hyperref}
\usepackage{verbatim}

\usepackage{times}
\usepackage{enumerate}
\usepackage[up,bf]{caption}

\input xy
\xyoption{all}

\newtheorem{theorem}{Theorem}[section]
\newtheorem{proposition}[theorem]{Proposition}

\newtheorem{lemma}[theorem]{Lemma}
\newtheorem{corollary}[theorem]{Corollary}

\theoremstyle{definition}
\newtheorem{definition}[theorem]{Definition}
\newtheorem{remark}[theorem]{Remark}

\newtheorem{problem}[theorem]{Problem}
\newtheorem{example}[theorem]{Example}

\numberwithin{equation}{section}
\numberwithin{figure}{section}

\newcommand\Acal{\mathcal{A}}

\newcommand\Fcal{\mathcal{F}}

\newcommand\Kcal{\mathcal{K}}

\newcommand\Mcal{\mathcal{M}}

\newcommand\Pcal{\mathcal{P}}

\newcommand\Tcal{\mathcal{T}}

\newcommand\bp{\mathbf{p}}
\newcommand\bq{\mathbf{q}}
\newcommand\bu{\mathbf{u}}
\newcommand\bv{\mathbf{v}}
\newcommand\bx{\mathbf{x}}
\newcommand\by{\mathbf{y}}
\newcommand\bw{\mathbf{w}}
\newcommand\bz{\mathbf{z}}

\newcommand\bN{\mathbf{N}}

\newcommand\Cscr{\mathscr{C}}

\newcommand\Fscr{\mathscr{F}}

\newcommand\Oscr{\mathscr{O}}

\newcommand\B{\mathbb{B}}
\newcommand\C{\mathbb{C}}
\newcommand\D{\overline{\mathbb D}}
\newcommand\CP{\mathbb{CP}}

\renewcommand\D{\mathbb D}

\renewcommand\H{\mathbb{H}}

\newcommand\N{\mathbb{N}}

\newcommand\R{\mathbb{R}}

\newcommand\T{\mathbb{T}}

\newcommand\igot{\mathfrak{i}}

\renewcommand\igot{\mathfrak{i}}

\newcommand\E{\mathrm{e}}

\renewcommand\imath{\igot}
\newcommand\zero{\mathbf{0}}

\newcommand\wt{\widetilde}
\newcommand\wh{\widehat}

\newcommand\dist{\mathrm{dist}}

\newcommand\tr{\mathrm{tr}}
\newcommand\Hess{\mathrm{Hess}}
\newcommand\Aut{\mathrm{Aut}}

\newcommand\CH{\mathrm{CH}}

\def\dist{\mathrm{dist}}

\newcommand\CK{\mathcal{CK}}

\numberwithin{equation}{section}

\begin{document}
\title{Hyperbolic domains in real Euclidean spaces}
\author{Barbara Drinovec Drnov\v sek and Franc Forstneri{\v c}}

\address{Barbara Drinovec Drnov\v sek, Faculty of Mathematics and Physics, University of Ljubljana, and Institute of Mathematics, Physics, and Mechanics, Jadranska 19, 1000 Ljubljana, Slovenia}
\email{barbara.drinovec@fmf.uni-lj.si}

\address{Franc Forstneri\v c, Faculty of Mathematics and Physics, University of Ljubljana, and Institute of Mathematics, Physics, and Mechanics, Jadranska 19, 1000 Ljubljana, Slovenia}
\email{franc.forstneric@fmf.uni-lj.si}

\thanks{The first named author was supported by the research program P1-0291 and grants J1-9104 and N1-0137, and N1-0237 from ARIS, Republic of Slovenia.
The second named author was supported by the European Union 
(ERC Advanced grant HPDR, 101053085) and the research program P1-0291 
and grants J1-3005, N1-0237, and N1-0237 from ARIS, Republic of Slovenia.} 

%
%
%
%
%
%

\date{2 December 2023}


\begin{abstract}
The second named author and David Kalaj introduced a pseudometric 
on any domain in the real Euclidean space $\R^n$, $n\ge 3$, 
defined in terms of conformal harmonic discs, by analogy with Kobayashi's pseudometric on 
complex manifolds, which is defined in terms of holomorphic discs. 
They showed that on the unit ball of $\R^n$, this minimal metric coincides with the 
classical Beltrami--Cayley--Klein metric. In the present 
paper we investigate properties of the minimal pseudometric
and give sufficient conditions for a domain to be (complete) hyperbolic, 
meaning that the minimal pseudometric is a (complete) metric. We show in particular that 
a convex domain is complete hyperbolic if and only if it does not contain any affine 2-planes.
One of our main results is that a domain with a negative minimal plurisubharmonic exhaustion 
function is hyperbolic, and a bounded strongly minimally convex domain is complete hyperbolic.
We also prove a localization theorem for the minimal pseudometric.
\end{abstract}

\maketitle

\centerline{\em Dedicated to H.\ Blaine Lawson in honour of his 80th birthday}

\tableofcontents

%
%
%
%
\section{Introduction}\label{sec:intro} 

A conformal structure on a surface is given by a smooth atlas whose charts are related 
by conformal diffeomorphisms of plane domains. A surface endowed with a conformal structure is 
a {\em conformal surface}. An oriented conformal surface is a Riemann surface, and 
a connected nonorientable conformal surface $M$ admits a two-sheeted conformal 
covering $\wt M\to M$ by a Riemann surface.
A conformal surface is {\em hyperbolic} if its universal conformal covering
space is the disc $\D=\{z\in\C:|z|<1\}$. Every hyperbolic surface carries a (Riemannian) 
{\em Poincar\'e metric}, $\Pcal_M$, such that any 
conformal universal covering $h:\D\to M$ is a local isometry from $(\D,\Pcal_\D)$
to $(M,\Pcal_M)$. The Poincar\'e metric on the disc is given by $\Pcal_\D(z,\xi)=\frac{|\xi|}{1-|z|^2}$
$(z\in\D,\ \xi\in\C)$. See \cite[Chapter 1]{AlarconForstnericLopez2021} for more details
on conformal surfaces and maps.

Let $ G_2(\R^n)$ denote the Grassman manifold of $2$-planes in $\R^n$.
On any domain $\Omega\subset \R^n$ $(n\ge 2)$, 
Forstneri\v c and Kalaj \cite{ForstnericKalaj2021} introduced a Finsler pseudometric 
$\Mcal_\Omega:\Omega\times  G_2(\R^n)\to\R_+=[0,+\infty)$, and an associated Finsler pseudometric 
$g_\Omega:\Omega\times\R^n\to \R_+$ and pseudodistance 
$\rho_\Omega:\Omega\times \Omega\to \R_+$ having the following two properties:
\begin{itemize}
\item every conformal harmonic map $M\to \Omega$ from a hyperbolic conformal surface, $M$, 
with its Poincar\'e metric $\Pcal_M$ is metric- and distance-decreasing, and
\item the pseudometric $g_\Omega$ and pseudodistance $\rho_\Omega$ are the largest ones
with this property.
\end{itemize}
Recall that a nonconstant conformal harmonic map $M\to\R^n$ $(n\ge 3)$ parameterizes 
a possibly branched minimal surface in $\R^n$, and every minimal surface 
with isolated singularities arises in this way 
(see \cite{AlarconForstnericLopez2021,Osserman1986,Tromba2012}). 
Hence, the pseudometric $g_\Omega$ and pseudodistance $\rho_\Omega$ describe the fastest 
possible rate of growth of conformal minimal surfaces in $\Omega$. 

The pseudometrics $\Mcal_\Omega$ and $g_\Omega$, and the pseudodistance $\rho_\Omega$
are defined in terms of conformal harmonic discs $\D\to \Omega$, 
by analogy with Kobayashi's definition of his pseudometric on complex manifolds in terms of 
holomorphic discs (see \cite{Kobayashi1967,Kobayashi1976,Kobayashi1998}).
We recall their construction and basic properties in Section \ref{sec:definition}. 
In Section \ref{sec:chains} we show that the pseudodistance 
$\rho_\Omega$ equals the infimum of the Poincar\'e lengths of chains 
of conformal harmonic discs in $\Omega$ connecting a given pair of points, 
in analogy to Kobayashi's definition of his pseudodistance by chains of holomorphic discs.
On domains in $\R^2=\C$, $g_\Omega$ and $\rho_\Omega$ agree 
with the Kobayashi pseudometric and pseudodistance, respectively.

The main part of the paper is devoted to developing geometric sufficient conditions for a 
domain $\Omega\subset\R^n$ $(n\ge 3)$ to be hyperbolic or complete hyperbolic, in the sense that 
$\rho_\Omega$  is a (complete) metric on $\Omega$; see Definition \ref{def:hyperbolic}. 
Note that a hyperbolic domain does not contain any minimal surfaces parameterized by $\C$
(such as a catenoid, a helicoid, the Enneper's surface, among others).
The main task in establishing (complete) hyperbolicity is to obtain suitable lower bounds 
on the minimal pseudometric $g_\Omega$. 

In Section \ref{sec:hyperbolic1} we show that a domain $\Omega\subset \R^n$ is hyperbolic
if and only if the minimal pseudometric $g_\Omega$ is bounded away from zero 
on a neighbourhood of any point, and this holds if and only if the family $\CH(\D,\Omega)$ of conformal
harmonic discs $\D\to\Omega$ is pointwise equicontinuous for some 
metric inducing the topology of $\Omega$ (see Theorem \ref{th:hyperbolic}).
These characterizations are analogous to Royden's results  
\cite[Theorem 2]{Royden1971} on Kobayashi hyperbolicity. 
We also establish a connection between (complete) hyperbolicity and tautness of a domain
$\Omega\subset \R^n$, where the latter condition means that $\CH(\D,\Omega)$ 
is a normal family; see Definition \ref{def:taut}.
The notion of a taut complex manifold was introduced by Wu \cite{Wu1967}, and
the analogous result for the Kobayashi metric is due to Kiernan \cite{Kiernan1970}.

In Section \ref{sec:convex} we show that a hyperbolic convex domain in $\R^n$ 
is also complete hyperbolic, and this holds if and only if the domain 
does not contain any affine $2$-planes, which is an obvious necessary condition
for hyperbolicity. The analogous results for Kobayashi hyperbolicity of convex domains 
in $\C^n$ are due to Barth \cite{Barth1980}, Harris \cite{Harris1979}, 
and Bracci and Saracco \cite{BracciSaracco2009}. Our proof differs
substantially from those in the cited works.

In Section \ref{sec:min-Kob} we show that if the tube $\Omega\times \imath\R^n\subset\C^n$
over a domain $\Omega\subset\R^n$ is Kobayashi (complete) hyperbolic, then $\Omega$ is 
(complete) hyperbolic; the converse fails in general.

In Section \ref{sec:Sibony} we introduce another pseudometric 
$\Fcal_\Omega:\Omega\times  G_2(\R^n)\to \R_+$ on any domain $\Omega\subset\R^n$ $(n\ge 3)$, 
defined by minimal plurisubharmonic functions on $\Omega$ and satisfying 
$\Fcal_\Omega\le \Mcal_\Omega$  (see Definition \ref{def:F} and Proposition \ref{prop:FM}).
Recall that an upper-semicontinuous function $u:\Omega\to[-\infty,+\infty)$ is said to be 
{\em minimal plurisubharmonic}, abbreviated MPSH, if and only if, for any conformal harmonic 
map $f:M\to\Omega$ from a conformal surface $M$, $u\circ f$ is a subharmonic function on $M$. 
We summarize basic properties of such functions in Section \ref{sec:Sibony} and
refer to \cite[Chapter 8]{AlarconForstnericLopez2021} for a more complete treatment. 
(Note that MPSH functions are a special case with $p=2$ of $p$-plurisubharmonic 
functions that were introduced and studied in different geometries by Harvey and Lawson 
\cite{HarveyLawson2011ALM,HarveyLawson2012AM,HarveyLawson2013IUMJ}.)
The pseudometric $\Fcal_\Omega$ is an analogue of the Sibony pseudometric on complex manifolds 
\cite{Sibony1981}, which is defined in terms of usual plurisubharmonic functions.
We show that most results for the Sibony pseudometric, obtained in \cite{Sibony1981},  
have their analogues for $\Fcal_\Omega$. In particular, a domain with a negative MPSH 
exhaustion function is hyperbolic (see Theorem \ref{th:mpsh}) and does not contain any 
conformal minimal surfaces of parabolic conformal type (see Corollary \ref{cor:hyperconvex}). 

In Section \ref{sec:localization} we obtain a localization result for the minimal pseudometric. 
Given a domain $\Omega\subset \R^n$ and a point $\bp\in b\Omega$ at which there exists a local 
negative MPSH function $u$ on $\overline \Omega$ peaking at $\bp$, with $u(\bp)=0$,  
we prove the estimate
\[
	g_\Omega(\bx,\cdotp) \ge \bigl(1-c|\bx-\bp|\bigr) \, g_{\Omega\cap\B(\bp,r_0)}(\bx,\cdotp)
\] 
with a constant $c=c(r_0)>0$, where $\B(\bp,r_0)$ is the open Euclidean ball of radius $r_0$
centred at $\bp$ (see Theorems \ref{th:loc1} and \ref{th:loc2}). 
In particular, the quotient of the two metrics converges to $1$ linearly in $|\bx-\bp|$ as $\bx\to\bp$.
Localization results for biholomorphically invariant metrics on complex manifolds
(such as those of Kobayashi, Sibony, and Azukawa) have a long history; see the papers \cite{Sibony1981,ForstnericRosay1987,Berteloot1994,Gaussier1999,Nikolov2002,IvashkovichRosay2004,DiederichFornaessWold2014,FornaessNikolov2021}, and the list is not exhaustive. 

Localization results are often useful in establishing (complete) hyperbolicity.
In Section \ref{sec:mpsh} we show that every bounded 
{\em strongly minimally convex domain} $\Omega$ in $\R^n$ is complete hyperbolic; 
see Theorem \ref{th:completehyperbolic}. 
(In $\R^3$, such domains are also called {\em strongly mean-convex} since they are
characterized by their boundaries having positive mean curvature at every point. 
Minimally convex domains play an important role in the theory of minimal surfaces, and 
we refer to \cite{DrinovecForstneric2016TAMS} 
and \cite[Chapter 8]{AlarconForstnericLopez2021} for surveys and further references. 
They are a particular case of $p$-convex domains, which were studied by Harvey and Lawson
\cite{HarveyLawson2011ALM,HarveyLawson2012AM,HarveyLawson2013IUMJ}.)
On a strongly minimally convex domain, 
the minimal metric $g_\Omega(\bx,\bv)$ is asymptotic to $|\bv|/\dist(\bx,b\Omega)$ 
in the radial direction and to $|\bv|/\sqrt {\dist(\bx,b\Omega)}$ in the tangential direction. 
Theorem \ref{th:completehyperbolic} is an analogue of the  
result of Graham \cite{Graham1975} that a bounded strongly pseudoconvex domain in 
$\C^n$ is Kobayashi complete hyperbolic; another proof 
using the Sibony metric is indicated in \cite[Proposition 7]{Sibony1981}. 
We could not adapt these proofs to our situation. 
Instead, we follow the approach developed by Ivashkovich and Rosay in \cite{IvashkovichRosay2004}
in the context of the Kobayashi metric on almost complex manifolds. 

In Section \ref{sec:extend} we prove a result on extending conformal harmonic maps
into hyperbolic domains across punctures. Theorem \ref{th:disc} is an analogue of a theorem 
of Kwack \cite{Kwack1969} (1969), which pertains to punctured holomorphic discs in Kobayashi
hyperbolic manifolds.

In Section \ref{sec:problems} we collect several open problems in this newly emerged field.

%
%
%
%
\section{Definition and basic properties of the minimal metric}\label{sec:definition}
To motivate the discussion and set the stage, we begin by recalling the definition of the 
Kobayashi pseudometric on a connected complex manifold $X$. 

We denote by $\Oscr(\D,X)$ the space of holomorphic maps $\D=\{z\in\C:|z|<1\}\to X$. 
Given a point $x\in X$, the Kobayashi pseudonorm of a tangent vector $\xi\in T_x X$ is 
\begin{equation}\label{eq:Kobayashi}
	\Kcal_X(x,\xi) = \inf\bigl\{1/r>0: \exists f\in \Oscr(\D,X),\ f(0)=x,\ f'(0)=r \xi\bigr\}\ge 0.
\end{equation}
The Kobayashi pseudodistance $\dist_X(x,y)\ge 0$ between a pair of points $x,y\in X$ 
is the infimum of the numbers
$
	\int_0^1 \Kcal(\gamma(t),\dot\gamma(t)) \, dt
$
over all piecewise smooth paths $\gamma:[0,1]\to X$ with $\gamma(0)=x$ and $\gamma(1)=y$.
A complex manifold $X$ is said to be (Kobayashi) {\em hyperbolic}
if $\dist_X$ is a metric on $X$, and 
{\em complete hyperbolic} if $(X,\dist_X)$ is a complete metric space.
Hyperbolicity clearly implies that there are no nonconstant holomorphic maps $\C\to X$. 
If $X$ is compact then the converse also holds according to Brody's theorem 
\cite{Brody1978}.
On the disc, the Kobayashi distance coincides with the Poincar\'e distance 
\[ 
	\dist_{\Pcal_\D}(z,w) = \frac12 \log\frac{1+d(z,w)}{1-d(z,w)},\quad
	\text{where}\ d(z,w)=\frac{|z-w|}{|1-z\overline w|}\ \text{for}\ z,w\in\D.
\] 
Any holomorphic map $X\to Y$ between complex manifolds is 
Kobayashi metric- and distance-decreasing, and this is the biggest 
pseudometric having this property which 
coincides on the disc $\D$ with the Poincar\'e metric. 

Suppose now that $\Omega$ is a connected domain in $\R^n$ for some $n\ge 3$.
We recall  from  \cite[Section 6]{ForstnericKalaj2021} the definition of the minimal pseudometric 
$g_\Omega$ and the minimal pseudodistance $\rho_\Omega$.
The point is to follow Kobayashi's idea but using conformal harmonic discs instead of holomorphic discs. 
The same construction applies in any Riemannian manifold.

A {\em conformal frame} in $\R^n$ is a pair of vectors $(\bu,\bv)\in \R^n\times \R^n$ such that 
\[ 
	|\bu|=|\bv|\quad \text{and}\quad \bu\,\cdotp \bv=0.
\] 
The dot stands for the Euclidean inner product on $\R^n$, and $|\bu|=\sqrt{\bu\,\cdot\bu}$ 
is the Euclidean norm of a vector $\bu\in\R^n$. 
Let $D$ be a domain in $\R^2$ with Euclidean coordinates $(x,y)$.
An immersion $f:D\to\R^n$ $(n\ge 2)$ is said to be {\em conformal} if its 
differential $df_p$ at any point $p\in D$ preserves angles. It is elementary 
(cf.\ \cite[Lemma 1.8.4]{AlarconForstnericLopez2021}) that this holds if and only if
\begin{equation}\label{eq:conformal}
	|f_x|=|f_y| \ \ \text{and}\ \ f_x\,\cdotp f_y=0.
\end{equation}
In other words, the partial derivatives $(f_x,f_y)$ are a conformal frame at every point.
A map $f:D\to\R^n$ of class $\Cscr^1$ 
(not necessarily an immersion) is called conformal if 
\eqref{eq:conformal} holds at each point. 
In the same way we define conformal maps $M\to\R^n$ from any
conformal surface, in particular from any Riemann surface.

We denote by $\CH(\D,\Omega)$ the space of (not necessarily immersed) 
conformal harmonic discs $\D\to \Omega$. 

%
%
We now introduce a couple of Finsler pseudometrics on a domain $\Omega\subset\R^n,\ n\ge 3$.
The first one, called the {\em minimal metric} and denoted $g_\Omega$, is defined on the tangent bundle $T\Omega=\Omega\times\R^n$; its value at a point $\bx\in\Omega$ and tangent vector 
$\bv\in\R^n$ is given by
\begin{equation}\label{eq:g}
	g_\Omega(\bx,\bv) = \inf\bigl\{1/r>0: \exists f\in\CH(\D,\Omega),\ f(0)=\bx,\ f_x(0)=r\bv\bigr\}.
\end{equation}
Clearly, $g_\Omega$ is upper-semicontinuous on $\Omega\times \R^n$ and absolutely homogeneous:  
\[
	g_\Omega(\bx,t\bv)=|t| g_\Omega(\bx,\bv)\ \ \text{for}\ \ t\in \R. 
\]
The {\em minimal pseudodistance} $\rho_\Omega:\Omega\times \Omega\to \R_+$ 
is obtained by integrating $g_\Omega$:
\begin{equation}\label{eq:rho}
	\rho_\Omega(\bx,\by)= \inf_\gamma \int_0^1 g_\Omega(\gamma(t),\dot\gamma(t)) \, dt,
	\quad\ \bx,\by\in\Omega.
\end{equation}
The infimum is over all piecewise smooth paths $\gamma:[0,1]\to\Omega$ with $\gamma(0)=\bx$
and $\gamma(1)=\by$. Obviously, $\rho_\Omega$ satisfies the triangle inequality, but it need not
be a distance function. In particular, $\rho_{\R^n}$ vanishes identically.
%
%
For every conformal harmonic disc $f:\D\to\Omega$ we have
\begin{equation}\label{eq:metricdecreasing1}
	g_\Omega(f(z),df_z(\xi)) \le \frac{|\xi|}{1-|z|^2} = 
	\Pcal_\D(z,\xi),\quad z\in \D,\ \xi\in\R^2,
\end{equation}
and $g_\Omega$ is the biggest pseudometric on $\Omega$ with this property.
For $z=0$ this follows from the definition of $g_\Omega$. For an arbitrary point
$z\in \D$ it is obtained by replacing $f$ by the conformal harmonic disc $f\circ\phi$, 
where $\phi\in\Aut(\D)$ is a holomorphic automorphism of the disc interchanging 
$0$ and $z$. (Explicitly, take $\phi(w)=\frac{z-w}{1-\bar z w}$ and note that $|\phi'(0)|=1-|z|^2$.)

It follows that conformal harmonic maps $M\to\Omega$ from any conformal surface 
are distance-decreasing in the Poincar\'e metric on $M$ and the minimal metric on $\Omega$:
\begin{equation}\label{eq:distancedecreasing1}
	\rho_\Omega(f(x),f(x')) \le \dist_{\Pcal_M}(x,x'),\quad\ x,x'\in M;
\end{equation}
furthermore, $\rho_\Omega$ is the biggest pseudodistance on $\Omega$ having 
this property. For the disc $\D$, this follows directly from \eqref{eq:metricdecreasing1}.
For other surfaces $M$, it is seen by passing to the universal conformal covering $\D\to M$
(see \cite[Proposition 6.1]{ForstnericKalaj2021}). If $M$ is not hyperbolic
then both sides of the inequality \eqref{eq:distancedecreasing1} vanish.

A map $R:\R^n\to\R^n$ is said to be a {\em rigid transformation} 
if it is a composition of an orthogonal map, a dilation,  
and a translation. If $R(\Omega)\subset \Omega'$ for a pair of domains in $\R^n$, then
\begin{equation}\label{eq:metricdecreasing2}
	g_{\Omega'}(R(\bx), dR_\bx(\bu)) \le 
	g_\Omega(\bx,\bu),\quad \bx\in\Omega,\ \bu\in\R^n.
\end{equation}
The same holds for any rigid embedding $R:\R^n\to \R^N$ for $3\le n<N$.
In particular, if $R(\Omega)=\Omega'$ then 
$R:(\Omega,g_\Omega)\to (\Omega',g_{\Omega'})$
is an isometry. Rigid transformations are the only self-maps of $\R^n$ taking 
any conformal minimal surface to another such surface.  
This is a marked difference from the Kobayashi pseudometric --- 
any holomorphic map between a pair of complex manifolds takes 
complex curves to complex curves, and hence is distance-decreasing. 
This difference reflects the fact that conformal minimal surfaces 
are much more abundant than holomorphic curves, and the space of such surfaces 
is preserved by a considerably smaller space of maps. This has major consequences when
estimating the minimal metric since  nonrigid coordinate changes are not allowed.

The second Finsler pseudometric is defined on $\Omega \times  G_2(\R^n)$,
where $ G_2(\R^n)$ denotes the Grassmann manifold of $2$-planes in $\R^n$. 
For a point $\bx\in\Omega$ and a $2$-plane $\Lambda\in  G_2(\R^n)$ we set 
\begin{equation}\label{eq:MLambda}
	\Mcal_\Omega(\bx,\Lambda) = 
	\inf \bigl\{ 1/\|df_0\|  : f\in \CH(\D,\Omega),\ f(0)=\bx,\ df_0(\R^2)=\Lambda\bigr\},
\end{equation}
where $\|df_0\|$ denotes the operator norm of the differential $df_0:\R^2\to\R^n$.  
It is immediate that $\Mcal_\Omega$ is upper-semicontinuous.
Comparing with the definition of $g_\Omega$ \eqref{eq:g}
it clearly follows that for any vector $\bv\in \R^n,\ |\bv|=1$ we have that 
\begin{equation}\label{eq:gM}
	g_\Omega(\bx,\bv) = 
	\inf\big\{ \Mcal_\Omega(\bx,\Lambda): \Lambda\in  G_2(\R^n), \bv\in\Lambda\}. 
\end{equation}
The set of $2$-planes $\Lambda\subset\R^n$ containing $\bv$ is parameterized 
by the sphere $S^{n-2}$ in the normal space $\bv^\perp\cong\R^{n-1}$.
This illuminates an important difference from the Kobayashi metric
on a domain in $ \C^n$: a given nonzero vector $ \bv\in\C^n$ determines a 
unique complex line $\C \bv$, and for the Kobayashi metric we only consider 
holomorphic discs tangent to that line. 

In Section \ref{sec:Sibony} we introduce another pseudometric 
$\Fcal_\Omega:\Omega\times  G_2(\R^n)\to \R_+$ defined by minimal 
plurisubharmonic functions on $\Omega$ and satisfying $\Fcal_\Omega\le \Mcal_\Omega$.
This lower bound helps us to establish several hyperbolicity results, a notion which we
now introduce. 

%
%
\begin{definition} \label{def:hyperbolic}
A domain $\Omega$ in $\R^n$ $(n\ge 3)$ is {\em hyperbolic} if $\rho_\Omega$ is a distance function 
on $\Omega$ (i.e., $\rho_\Omega(\bx,\by)>0$ for any pair of distinct points $\bx,\by\in\Omega$), 
and is {\em complete hyperbolic} if $(\Omega,\rho_\Omega)$ is a complete metric space. 
\end{definition}

%
%
\begin{example}[Minimal metric on the ball]\label{ex:ball}
It was shown by Forstneri\v c and Kalaj \cite[Theorem 6.2]{ForstnericKalaj2021} that the metric $g_{\B^n}$ 
on the unit ball $\B^n\subset \R^n$ $(n\ge 3)$ equals the {\em Beltrami--Cayley--Klein metric}
whose value at any point $\bx\in\B^n$ and tangent vector $\bv\in\R^n$ equals
\begin{equation}\label{eq:CK}
	\CK(\bx,\bv)^2 = \frac{(1-|\bx|^2)|\bv|^2 + |\bx\,\cdotp \bv|^2}{(1-|\bx|^2)^2}
	= \frac{|\bv|^2}{1-|\bx|^2} + \frac{|\bx\,\cdotp \bv|^2}{(1-|\bx|^2)^2}. 
\end{equation}
This metric is complete, and it agrees with the restriction of the Kobayashi metric 
(and, up to a multiplicative constant, of the Bergman metric) on the unit ball $\B^n_\C$ in $\C^n$
to points of the real ball $\B^n=\B^{n}_\C\cap \R^n$ and real tangent vectors. 
The extremal conformal minimal discs in $\B^n$ are conformal parameterizations
of affine discs, intersections of $\B^n$ with affine 2-planes.
The operator norm $\|df_z\|$ of the differential of such a disc equals $1$ for every $z\in\D$
in this pair of metrics. However, the extremal discs are not isometries: 
by \cite[Theorem 2.6]{ForstnericKalaj2021} 
the equality holds in \eqref{eq:metricdecreasing1} for some $(z,\xi)\in\D\times (\R^2\setminus\{0\})$ 
if and only if $f$ is a conformal diffeomorphism onto the affine disc $\Sigma=(f(z)+df_z(\R^2))\cap \B^n$ 
and the vector $df_z(\xi)$ is tangent to the diameter of $\Sigma$ through the point $f(z)$.
This differs from the complex case where the extremal holomorphic discs in 
$\B^n_\C$  are the affine complex discs, and their conformal parameterizations are isometries 
in the Poincar\'e metric on $\D$ and the Kobayashi metric on $\B^n_\C$. 

The ball seems to be the only hyperbolic domain in $\R^n$
for which an explicit expression for the minimal metric is known.
Considering the ball $\B^{2n}\subset\R^{2n}$ as a complex ball $\B^n_\C\subset \C^n$,
we have that $g_{\B^{2n}} \le \Kcal_{\B^n_\C}$ but the two metrics are not equal. 
By \cite[Corollary 2.3]{ForstnericKalaj2021} the Kobayashi extremal discs in $\B^n_\C$
are precisely those extremal conformal minimal discs which are complex, namely, 
the proper affine complex discs in $\B^n_\C$.
\qed\end{example}

%
%
By definition, $\rho_\Omega$ is an {\em inner pseudodistance}, 
i.e., the pseudodistance between any pair of points equals the infimum of $g_\Omega$-lengths 
of curves connecting them. Hence, the Hopf--Rinow theorem \cite{HopfRinow1931} 
(see also Jost \cite[Theorem 1.4.8]{Jost2017}) implies the following;
cf.\ Royden \cite[Proposition 7]{Royden1971} and Abate \cite[Proposition 2.3.17]{Abate1989}
for the complex case.

\begin{proposition}\label{prop:HRM}
A hyperbolic domain $\Omega\subset \R^n$ is complete 
hyperbolic if and only if for every point $\bp\in \Omega$ and number $r>0$ the closed 
ball $\{\bx\in \Omega: \rho_\Omega(\bp,\bx)\le r\}$ is compact.
\end{proposition}

%
%
When estimating the minimal pseudometric, one may use the {\em comparison principle}, 
which says that for any triple of domains in $\R^n$ we have that 
\begin{equation}\label{eq:comparison}
	\Omega_0\subset \Omega_1 \ \Longrightarrow \ 
 	g_{\Omega_1} \le g_{\Omega_0}\ \text{holds on $\Omega_0$}.
\end{equation} 
This is obvious from the definition of $g_\Omega$. The corresponding inequalities also
hold for the respective pseudometrics $\Mcal$ and pseudodistances $\rho$
on the two domains. Together with Example \ref{ex:ball} we obtain: 

%
%
\begin{corollary}\label{cor:bounded}
Every bounded domain in $\R^n$ $(n\ge 3)$ is hyperbolic.
\end{corollary}

The following example shows that there are bounded domains which are not complete hyperbolic.
A domain $\Omega$ in this example clearly cannot be convex. In fact, we will show that every 
hyperbolic convex domain is also complete hyperbolic; see Theorem \ref{th:convex}.

%
%
\begin{example}\label{ex:notcompletehyperbolic}
Let $\D^*=\D\setminus\{0\}$. 
Assume that $\Omega\subset\R^n$ is a connected domain whose boundary 
is of class $\Cscr^1$ near a point $\bp\in b\Omega$.
If there exists a conformal harmonic disc $f:\D\to\R^n$ such 
that $f(0)=\bp$ and $f(\D^*)\subset \Omega$, then $\bp$ is at finite minimal distance 
from $\Omega$. 

Since the Poincar\'e metric on the punctured disc $\D^*$ 
is complete at the origin (see \cite[p.\ 78]{Kobayashi2005}),
we cannot make this conclusion by using the fact that $f:\D^*\to \Omega$ 
is distance-decreasing. Instead, we proceed as follows.
Let $\bN$ denote the inward unit normal vector to $b\Omega$ at $\bp$. Consider the family of conformal
harmonic discs $f_t(z)=f(z)+t\bN\in \R^n$ for $z\in\D$ and $t\ge 0$. There are constants 
$c>0$ and $0<\delta<1$ such that $f_t(\delta \overline \D)\subset \Omega$ for all $0<t\le c$. 
Replacing $f(z)$ by $f(\delta z)$ we may assume that $\delta=1$. 
Choose a point $a\in\D\setminus \{0\}$, so
$f(a)\in \Omega$. Given $\epsilon\in (0,c)$, we consider the path 
$\gamma_\epsilon:[0,1] \to \Omega$ defined as follows:
\[
	\gamma_\epsilon(t) =
	\begin{cases} f(a)+(c+ 2t(\epsilon-c))\bN,  & 0 \le t\le 1/2;\\
	                       f(2(1-t)a)+\epsilon\bN, & 1/2\le t\le 1.
	\end{cases}
\]
Note that $\gamma_\epsilon(0)=f(a)+c\bN$ and $\gamma_\epsilon(1)=\bp+\epsilon\bN$. 
The minimal length of the first part of $\gamma_\epsilon$ for $0\le t\le 1/2$ 
is smaller than the minimal length of the compact arc $\{f(a)+s\bN:0\le s\le c\}\subset \Omega$, 
which is a finite number $C>0$ independent of $\epsilon$. 
Since a conformal harmonic disc is metric-decreasing \eqref{eq:metricdecreasing1}, 
the minimal length of the second part of $\gamma_\epsilon$ is less than or equal to 
the Poincar\'e distance in $\D$ from $0$ to $a$, which is $\frac12 \log\frac{1+|a|}{1-|a|}$. 
This shows that $\bp$ is contained in the closure of the minimal ball in $\Omega$ of radius
$C+\frac12 \log\frac{1+|a|}{1-|a|}$ centred at $f(a)+c\bN\in \Omega$.
By Proposition \ref{prop:HRM}, $\Omega$ fails to be complete hyperbolic at $\bp$.
\qed\end{example}

By the comparison principle \eqref{eq:comparison}, an inscribed ball in a domain $\Omega\subset\R^n$ 
provides an upper bound on $\Mcal_\Omega$ and $g_\Omega$. 
If $\Omega$ has $\Cscr^2$ boundary then by using a family of inscribed osculating balls, 
we obtain upper bounds on $\Mcal_\Omega$ and $g_\Omega$, which are asymptotically optimal up to 
a multiplicative constant when the point approached the boundary $b\Omega$. 

The main problem when trying to establish (complete) hyperbolicity 
is to obtain suitable lower bounds on $g_\Omega$. By using circumscribed balls, one can show 
that every bounded strongly convex domain $\Omega\subset\R^n$  is complete hyperbolic, 
and asymptotic estimates for the minimal metric $g_\Omega$ follow from the formula \eqref{eq:CK} 
for the minimal metric on the ball (see \cite[Theorem 6.9]{ForstnericKalaj2021}). 
More generally, we show in Lemma \ref{lem:halfspace} that the minimal distance to any affine
hyperplane in $\R^n$ is infinite, which leads to a characterization of (possibly unbounded)
convex domains which are complete hyperbolic; see Theorem \ref{th:convex}. 
However, the comparison principle does not provide good lower bounds on nonconvex domains,
and we shall develop more powerful methods to deal with this issue.

%
%
%
We conclude this section by comparing the minimal metric $g_\Omega$
on a domain $\Omega$ in an even dimensional Euclidean space $\R^{2n}$ $(n\ge 2)$
with the Kobayashi metric $\Kcal_\Omega$ when considering $\Omega$ as a domain in $\C^n$.
Since every holomorphic disc is also conformal harmonic, we see that $\Kcal_\Omega\ge g_\Omega$. 
In particular, if $\Omega$ is minimal (complete) hyperbolic then it is Kobayashi (complete) hyperbolic. 
The converse fails as the following example shows.

\begin{example}
Consider $\D$ as the unit disc in $\R^2$, identified with the real subspace of $\C^2$. 
The domain $\Omega = \D\times\imath\R^2$ is not minimal hyperbolic since it 
contains affine planes, but it is Kobayashi hyperbolic by Barth \cite{Barth1980} 
since it does not contain any affine complex line.
\qed\end{example}

%
%
%
%
\section{The minimal distance is given by chains of conformal harmonic discs}\label{sec:chains}

In this section we show that the minimal pseudodistance $\rho_\Omega$ \eqref{eq:rho} 
can be defined by chains of conformal harmonic discs in $\Omega$. This definition of 
$\rho_\Omega$ was introduced in \cite[Section 6]{ForstnericKalaj2021},
and here we show that it coincides with the integral of the minimal pseudometric 
$g_\Omega$ \eqref{eq:g}.
We follow the spirit of Kobayashi's original definition  in \cite{Kobayashi1967} of his pseudodistance 
on complex manifolds, replacing holomorphic discs by conformal harmonic ones. 
The proof of the desired equality is similar to Royden's proof \cite[Theorem 1]{Royden1971} 
for the Kobayashi pseudometric.  
See also Kruglikov \cite{Kruglikov1999} for the corresponding result on almost complex manifolds.

Let $\Omega$ be a domain in $\R^n$ for $n\ge 3$.
Fix a pair of points $\bx,\by\in \Omega$. Given a finite chain of conformal harmonic 
discs $f_i:\D\to \Omega$ and points $a_i\in \D$ $(i=1,\ldots,k)$ such that 
\begin{equation}\label{eq:chain}
	f_1(0)=\bx,\quad f_{i+1}(0)=f_i(a_i)\ \text{for $i=1,\ldots,k-1$},\quad f_k(a_k)=\by,
\end{equation}
we define the length of the chain to be the number
\begin{equation}\label{eq:lengthofchain}
	\sum_{i=1}^k \dist_{\Pcal_\D}(0,a_i) = \sum_{i=1}^k \frac{1}{2}\log \frac{1+|a_i|}{1-|a_i|} \ge 0.
\end{equation}
We define the pseudodistance 
\begin{equation}\label{eq:tau}
	\tau_\Omega(\bx,\by)\ge 0,\quad \ \bx,\by\in \Omega
\end{equation}
as the infimum of the lengths \eqref{eq:lengthofchain} of all chains 
as in \eqref{eq:chain} (cf.\ \cite[Sect.\ 6]{ForstnericKalaj2021}).
Note that the Kobayashi pseudodistance on a complex manifold $X$ is defined in 
the same way by using chains of holomorphic discs $\D\to X$ (see Kobayashi \cite{Kobayashi1967,Kobayashi1976}).

We denote by $\B(\bp,\delta)$ the open Euclidean ball of radius $\delta>0$ centred at a 
point $\bp\in\R^n$, and by $\overline \B(\bp,\delta)$ its closure. By 
\cite[Theorem 6.2]{ForstnericKalaj2021} we have 
\[
	\tau_{\B^n} = \rho_{\B^n} =\dist_{\CK},
\]
where $\dist_{\CK}$ is the distance function induced by Beltrami--Cayley--Klein metric
\eqref{eq:CK}. On other Euclidean balls we obtain the minimal distance
by translation and dilation of coordinates.
Hence, for every domain $\Omega$, point $\bx_0\in\Omega$, 
and number $\delta>0$ such that 
$\overline \B(\bx_0,\delta)\subset\Omega$ there is a $C>0$ such that 
\begin{equation}\label{eq:tauvseuc}
	\tau_{\Omega}(\bx,\by)\le C |\bx-\by| \ \text{ for every }\bx,\,\by\in \B(\bx_0,\delta).
\end{equation}

%
%
\begin{theorem}\label{th:rhotau}
For any domain $\Omega\subset\R^n$ $(n\ge 3)$ we have that 
\[
	\rho_\Omega = \tau_\Omega,
\]
where 
$\rho_\Omega$ and $\tau_\Omega$ are defined by \eqref{eq:rho} and \eqref{eq:tau},
respectively.
\end{theorem}

\begin{proof}
We follow Royden's proof of the analogous result for the Kobayashi pseudodistance
(see \cite[Theorem 1, p.\ 130]{Royden1971}). Let $f_i:\D\to \Omega$ and $a_i\in \D$ $(i=1,\ldots,k)$
be as in \eqref{eq:chain}. 
Since conformal harmonic discs $\D\to\Omega$ are distance-decreasing from the Poincar\'e distance  
on $\D$ to the pseudodistance $\rho_\Omega$ (see \eqref{eq:distancedecreasing1}), we have that 
\[
	\rho_\Omega(\bx,\by)\le \sum_{i=1}^k  \rho_\Omega(f_i(0),f_i(a_i))\le 
	\sum_{i=1}^k \dist_\Pcal(0,a_i).
\]
Taking the infimum over all finite chains as in \eqref{eq:chain} 
gives $\rho_\Omega(\bx,\by)\le \tau_\Omega(\bx,\by)$.

To prove the converse inequality, fix $\epsilon >0$ and choose a smooth path 
$\gamma:[0,1]\to\Omega$ with $\gamma(0)=\bx$ and $\gamma(1)=\by$ satisfying
\[
	\int_0^1 g_\Omega(\gamma(t),\dot\gamma(t)) dt <  \rho_\Omega(\bx,\by) + \epsilon.
\]
Since $g_\Omega$ is upper-semicontinuous, there is a continuous function 
$h: [0,1]\to \R$ with $h(t)>g_\Omega(\gamma(t),\dot\gamma(t)) $ for all $t\in[0,1]$ and such that
$\int_0^1 h(t) dt < \rho_\Omega(\bx,\by) + \epsilon$. Then for every sufficiently fine partition 
$0=t_0<t_1<\cdots <t_n=1$ of $[0,1]$ we have that 
\[
	\sum_{i=1}^n h(t_{i-1})(t_i-t_{i-1})<  \rho_\Omega(\bx,\by) + \epsilon.
\]
By (\ref{eq:tauvseuc}) there are constants $C>0$ and $\delta>0$ 
such that 
\begin{equation}\label{eq:tauvseuc2}
	\tau_\Omega(\bx,\by)\le C |\bx-\by| 
	\text{ for every }\bx,\,\by\in \B(\gamma(t),\delta)\ \text{and}\ t\in[0,1].
\end{equation}
Fix $t\in[0,1]$. Since  $h(t)>g_\Omega(\gamma(t),\dot\gamma(t))$, there are a map $f\in \CH(\D,\Omega)$
and a number $r>0$ such that $f(0)=\gamma(t)$, $f_x(0) = r\dot\gamma(t)$, and $h(t)>1/r$.
Then for any $s\in\R$ close to $0$ we have 
$
	f(s/r)=\gamma(t)+\frac{s}{r}f_x(0)+O(s^2)=\gamma(t)+s\dot\gamma(t)+O(s^2), 
$
and hence
\begin{eqnarray*}
	\tau_\Omega(\gamma(t),\gamma(t)+s\dot\gamma(t))&\le& 
	 \tau_\Omega(f(0),f(s/r)) + \tau_\Omega(f(s/r),\gamma(t)+s\dot\gamma(t))  \\
	&\le& \dist_\Pcal(0,s/r) + C|f(s/r) - \gamma(t) - s\dot\gamma(t)| \\
	&\le& |s|/r+O(|s|^2)< |s|\cdot h(t)+O(|s|^2).
\end{eqnarray*}
(We also used that $\dist_\Pcal(0,a)=|a|+O(|a|^2)$ for $a\in\D$ close to $0$.)
Therefore for $t,t'\in[0,1]$ close enough we obtain
\begin{eqnarray*}
\tau_\Omega(\gamma(t),\gamma(t'))&\le&
\tau_\Omega(\gamma(t),\gamma(t)+(t'-t)\dot\gamma(t))+
\tau_\Omega(\gamma(t)+(t'-t)\dot\gamma(t),\gamma(t'))\\
&\le&|t'-t|\cdot h(t)+O(|t'-t|^2),
\end{eqnarray*}
where the second term on the right hand side was estimated by \eqref{eq:tauvseuc2}. 
This implies that for all sufficiently fine partitions $0=t_0<t_1<\cdots <t_n=1$ of $[0,1]$ we have that 
\begin{eqnarray*}
\tau_\Omega(\bx,\by)&\le&\sum_{i=1}^n \tau_\Omega(\gamma(t_{i-1}),\gamma(t_i))\\
&\le&\sum_{i=1}^n (t_{i}-t_{i-1})h(t_{i-1})(1+\epsilon)<(1+\epsilon) (\rho_\Omega(\bx,\by) + \epsilon).
\end{eqnarray*}
Since $\epsilon>0$ was arbitrary, this concludes the proof.
\end{proof}

%
%
%
%
\section{Characterizations of hyperbolic domains}\label{sec:hyperbolic1} 

Recall that $g_\Omega$ denotes the pseudometric \eqref{eq:g} 
on a domain $\Omega\subset\R^n$.  The following concept mimics Royden's notion of hyperbolicity 
in a complex manifold (see \cite[p.\ 133]{Royden1971}).

\begin{definition}\label{def:Rhyperbolic}
A domain $\Omega\subset\R^n$ $(n\ge 3)$ is {\em hyperbolic at a point $\bp\in \Omega$} if 
there are a neighbourhood $U\subset \Omega$ of $\bp$ and a constant $c>0$ such that
\[
	g_\Omega(\bx,\bu)\ge c |\bu|\ \ \text{for all $\bx\in U$ and $\bu\in\R^n$.}
\]
The domain $\Omega$ is {\em R-hyperbolic} if it is hyperbolic at every point.
\end{definition}

Recall (see Kelley \cite{Kelley1975}) that a family $\Fscr$ of mappings of a topological space $X$ 
to a topological space $Y$ is said to be an {\em even family} if, given points $x\in X$ and $y\in Y$ and 
a neighbourhood $U\subset Y$ of $y$, there are a neighbourhood $V\subset X$ of $x$
and a neighbourhood $W\subset U$ of $y$ such that for every $f\in \Fscr$
with $f(x)\in W$ we have that $f(V)\subset U$. 

The following is an analogue of Royden's result \cite[Theorem 2, p.\ 133]{Royden1971}, 
which pertains to the Kobayashi pseudodistance on a complex manifold.

%
%
%
\begin{theorem}\label{th:hyperbolic}
The following conditions are equivalent for a domain $\Omega\subset\R^n$.
\begin{enumerate}[\rm (i)]
\item The family $\CH(\D,\Omega)$ of conformal harmonic discs $\D\to\Omega$ 
is pointwise equicontinuous for some metric $\rho$ on $\Omega$ inducing 
the natural topology.
\item The family $\CH(\D,\Omega)$ is an even family.
\item The domain $\Omega$ is R-hyperbolic (see Definition \ref{def:Rhyperbolic}).
\item The domain $\Omega$ is hyperbolic in the sense of Definition \ref{def:hyperbolic}.
\item The distance function $\rho_\Omega$ induces the topology of $\Omega$.
\end{enumerate}
\end{theorem}

%
%
\begin{remark}\label{rem:tighttaught}
Following the terminology introduced for complex manifolds and holomorphic maps $\D\to\ \Omega$
by Wu \cite{Wu1967} and Royden \cite[p.\ 133]{Royden1971},  
a domain $\Omega\subset \R^n$ endowed with a metric $\rho$ inducing the natural 
topology of $\Omega$ is called {\em tight} if condition (i) holds. Hence:
\[
	(\Omega,\rho)\ \text{is tight}\ \Longrightarrow\ \Omega\ \text{is hyperbolic}
	\ \Longrightarrow\ (\Omega,\rho_\Omega)\ \text{is tight}.
\]
\end{remark}

\begin{proof}[Proof of Theorem \ref{th:hyperbolic}]
We follow Royden's proof in \cite[Theorem 2]{Royden1971} 
of the corresponding results for the Kobayashi pseudodistance. Similar arguments were given by 
Kiernan \cite{Kiernan1970}.

Since every pointwise equicontinuous family of maps with respect to some metric on $\Omega$ inducing the 
topology of $\Omega$ is an even family (see Kelley \cite[p.\ 237]{Kelley1975}), (i) implies (ii). 

Assume that (ii) holds. 
Choose $\bx\in\Omega$ and $\epsilon>0$ such that
$\overline \B(\bx,\epsilon)\subset \Omega$. Since $\CH(\D,\Omega)$ is an even family, there are
a number $0<\delta<1$ and a neighbourhood $W\subset \B(\bx,\epsilon)$ of $\bx$ such that for 
any $f\in \CH(\D,\Omega)$ with $f(0)\in W$ we have $f(\delta \D)\subset \B(\bx,\epsilon)$. 
(Here, $\delta\D=\{z=x+\imath y\in\C: |z|<\delta\}$.) Fix a point $\by\in W$ and a vector $\bu\in\R^n$.
Then, for any $f\in\CH(\D,\Omega)$ such that $f(0)=\by$ and $f_x(0)=r\bu$ for some $r>0$, the map $h(z)=f(\delta z)$ $(z\in \D)$ is conformal harmonic, $h$ maps $\D$ into $\B(\bx,\epsilon)$,
$h(0)=\by$, and $h_x(0)=\delta r\bu$. This implies that
$\delta g_{\B(\bx,\epsilon)}(\by,\bu)\le g_\Omega(\by,\bu)$.
Since the metric $g_\B$ on the unit ball $\B\subset\R^n$ 
agrees with the Beltrami--Cayley--Klein metric (see Example \ref{ex:ball}), 
there is a $C>0$ such that $g_{\B(\bx,\epsilon)}(\by,\bu)\ge C|\bu|$ for all $\by\in W$. 
Hence $\Omega$ is R-hyperbolic at $\bx$, which implies (iii).

The implication (iii) $\Rightarrow$ (iv) is an immediate consequence of the definition of
R-hyperbolicity (see Definition \ref{def:Rhyperbolic}) and the definition \eqref{eq:rho}
of the pseudodistance $\rho_\Omega$. 

Assume that (iv) holds. By \eqref{eq:tauvseuc} and Theorem \ref{th:rhotau} the topology induced by 
$\rho_\Omega$ is weaker than the standard topology on $\Omega$. To see that they are equivalent, 
fix a point $\bx\in \Omega$ and a number $\epsilon>0$ such that $\B(\bx,2\epsilon)\subset \Omega$.
By \eqref{eq:tauvseuc} and Theorem \ref{th:rhotau}
the function $\by\mapsto \rho_\Omega(\bx,\by)$ is continuous in the standard topology.
Therefore, there is  a point $\by_0$ with $|\bx-\by_0|=\epsilon$ such that
\[
	\inf\{\rho_\Omega(\bx,\by): |\bx-\by|=\epsilon \}=\rho_\Omega(\bx,\by_0).
\]
By (iv) the number $\delta=\rho_\Omega(\bx,\by_0)$ is positive.
For any $\bz\in \Omega \setminus \B(\bx,\epsilon)$ each smooth path 
$\gamma$ connecting $\bx$ and $\bz$ crosses the boundary $b\B(\bx,\epsilon)$. 
Denote the first such point by $\by$. Then the $\rho_\Omega$-length of $\gamma$ 
is larger than or equal to the length of the restriction of $\gamma$ to the part
connecting $\bx$ and $\by$, which is at least $\delta$. Taking the infimum over all paths connecting 
$\bx$ and $\bz$ we get that $\rho_\Omega(\bx,\bz)\ge \delta$. This gives 
$\{\bz\in\Omega:\rho_\Omega(\bx,\bz)<\delta\}\subset \B(\bx,\epsilon)$, so the topology 
defined by $\rho_\Omega$ is stronger than the standard topology. This implies (v).

Finally, the implication (v) $\Rightarrow$ (i)  is a consequence of the distance-decreasing 
property conformal harmonic discs $\D\to \Omega$, see \eqref{eq:distancedecreasing1}. 
\end{proof}

%
%
\begin{corollary} \label{cor:comparablemetric}
The following are equivalent for a domain $\Omega\subset \R^n$, $n\ge 3$.
\begin{enumerate}[\rm (a)] 
\item The domain $\Omega$ is hyperbolic.
\item The metric $g_\Omega$ is comparable to the Euclidean metric $ds^2=dx_1^2+\cdots+dx_n^2$ 
on any compact subset of $\Omega$.
\item The distance $\rho_\Omega$ is comparable to the Euclidean distance
on any  compact subset of $\Omega$.
\end{enumerate}
\end{corollary}

\begin{proof}
Assume that (a) holds. Pick a point $\bp\in \Omega$. Let $B'\Subset B\subset \Omega$ be
a pair of balls centred at $\bp$. Then $g_\Omega\le g_B$ holds on $B$ by the comparison principle. 
Since $g_B$ is a Riemannian metric (see Example \ref{ex:ball}),
it is comparable to the Euclidean metric $ds^2$ on any smaller ball $B'$,
and hence $g_\Omega\le c' ds^2$ on a neighbourhood of $\bp$ for some constant $c'>0$.  (This holds on any domain.)
Since $\Omega$ is hyperbolic, it is $R$-hyperbolic by Theorem \ref{th:hyperbolic}, 
which means that $g_\Omega\ge c ds^2$ on a neighbourhood of $\bp$ 
for some $c>0$ (see Definition \ref{def:Rhyperbolic}). This proves (a) $\Rightarrow$ (b).
The implication (b) $\Rightarrow$ (c) is obvious by integration
and using the argument in the proof of the implication (iii) $\Rightarrow$ (iv) in 
Theorem \ref{th:hyperbolic}. The implication  (c) $\Rightarrow$ (a) is trivial.
\end{proof}

%
%
\begin{corollary}\label{cor:equicontinuous}
Let $\Omega$ be a hyperbolic domain in $\R^n$ $(n\ge 3)$. For every conformal
surface, $M$, the set $\CH(M,\Omega)$ is pointwise equicontinuous, and hence an even family. 
\end{corollary}

\begin{proof}
If $M$ is not hyperbolic (i.e., its universal conformal covering space is $\C$),
then every conformal harmonic map $M\to\Omega$ into a hyperbolic domain $\Omega$ is constant. 
Assume now that $M$ is hyperbolic and let $\dist_\Pcal$ denote the Poincar\'e distance
on $M$. Then,  every conformal harmonic map $f:M\to \Omega$ satisfies 
$\rho_\Omega(f(x),f(x'))\le \dist_{\Pcal_M}(x,x')$ for any pair of points $x,x'\in M$
(see \eqref{eq:distancedecreasing1}), so 
$\CH(M,\Omega)$ is uniformly Lipschitz in this pair of metrics.   
By Corollary \ref{cor:comparablemetric}, $\rho_\Omega$ is locally comparable
to the Euclidean metric. It follows that $\CH(M,\Omega)$ is pointwise
equicontinuous, and hence an even family by Kelley \cite[p.\ 237]{Kelley1975}.
\end{proof}

%
%
Recall that a complex manifold $\Omega$ is called {\em taut} if the family $\Oscr(\D,\Omega)$  
of holomorphic discs $\D\to\Omega$ is a {\em normal family} (see Wu \cite{Wu1967} and Royden 
\cite[p.\ 135]{Royden1971}). We introduce the corresponding notion for domains in $\R^n$ 
by using conformal harmonic discs.

\begin{definition}\label{def:taut}
Let $\Omega$ be a domain in $\R^n$.
\begin{enumerate}[\rm (i)]
\item A sequence $(f_i)_{i\in\N}$ in $\Cscr(\D,\Omega)$ is \emph{compactly divergent} if for any pair of 
compact sets $K\subset \D$ and $K'\subset \Omega$, $f_i(K)\cap K'$ is empty 
for all but finitely many indices $i\in\N$.
\item A family of maps $\Fscr\subset \Cscr(\D,\Omega)$ is \emph{normal}
if each sequence  in $\Fscr$ has a subsequence that is either uniformly convergent on
compact sets in $\Omega$ or compactly divergent.
\item The domain $\Omega$ is \emph{taut} if the family $\CH(\D,\Omega)$ 
of conformal harmonic discs  $\D\to\Omega$ is a normal family. 
\end{enumerate}
\end{definition}

%
%
\begin{remark}\label{rem:taut}
Tautness of a domain $\Omega$ is equivalent to the family 
$\CH(\D,\Omega)$ being locally compact in the compact-open topology.
If $\Omega$ is taut, then for any connected open Riemann surface, $M$, the space $\CH(M,\Omega)$ 
of conformal harmonic maps $M\to\Omega$ is also locally compact in the compact-open topology
(a normal family). In fact, given a point $p\in M$ and a compact set $K\subset \Omega$, 
the set $\{f\in \CH(M,\Omega):f(p)\in K\}$ is compact.
\qed\end{remark}

The next theorem is an analogue of the corresponding result for the Kobayashi metric,
due to Kiernan \cite{Kiernan1970} and Royden \cite{Royden1971}.

%
%
\begin{theorem}\label{thm:tauthyp}
The following hold for any domain $\Omega$ in $\R^n$, $n\ge 3$:
\begin{enumerate}[\rm (i)]
\item If $\Omega$ is complete hyperbolic, then it is taut.
\item If $\Omega$ is taut, then it is hyperbolic.
\end{enumerate}
\[
	\text{complete hyperbolic}\ \Longrightarrow \ \text{taut}
	\ \Longrightarrow \ \text{hyperbolic}
\]
\end{theorem}

\begin{proof}
Part (i)  is a consequence of the Hopf--Rinow theorem 
\cite{HopfRinow1931}, \cite[Theorem 1.4.8]{Jost2017};
see Proposition \ref{prop:HRM} and Royden \cite[Corollary, p.\ 136]{Royden1971}.
It also follows from the result of Kobayashi \cite[Theorem 3.1]{Kobayashi2005}, 
which pertains to more general situations.
Part (ii) is proved along the same lines as for the Kobayashi pseudodistance, see
Kiernan \cite[Proposition 2]{Kiernan1970}. 
\end{proof}

%
%
%
%
\section{Hyperbolicity of convex domains}\label{sec:convex}

In this section we provide a geometric characterization of hyperbolic convex domains in $\R^n$. 
The corresponding result for Kobayashi hyperbolicity 
of convex domains in $\C^n$ is due to Barth \cite{Barth1980}, 
Harris \cite[Theorem 24]{Harris1979}, and Bracci and Saracco
\cite[Theorem 1.1]{BracciSaracco2009}. 

%
%
\begin{theorem}\label{th:convex}
The following conditions are equivalent for a 
convex domain $\Omega\subset\R^n$, $n\ge 3$.
\begin{enumerate}[\rm (i)]
\item The domain $\Omega$ is complete hyperbolic.
\item For any open Riemann surface, $M$, the family $\CH(M,\Omega)$ of conformal harmonic maps
$M\to\Omega$ is a normal family. In particular, $\Omega$ is taut (see Definition \ref{def:taut}).
\item The domain $\Omega$ is hyperbolic.
\item The domain $\Omega$ does not contain any $2$-dimensional affine subspaces.
\item The domain $\Omega$ has $n-1$ linearly independent separating hyperplanes.
\end{enumerate}
\end{theorem}

Recall that a hyperplane $\Sigma\subset\R^n$ is called {\em separating} for a domain $\Omega\subset\R^n$
if $\Sigma\cap\Omega=\varnothing$, so $\Omega$ lies in one of the two half-spaces in 
$\R^n\setminus \Sigma$. 

It is a classical result of convexity theory that a convex domain $\Omega\subset\R^n$ 
containing an affine line $L$ is the product $\Omega=D\times L$, where $D$ is a convex
domain in a hyperplane
perpendicular to $L$. Theorem \ref{th:convex} says that such $\Omega$ 
is (complete) hyperbolic if and only if $D$ does not contain any affine lines.

\begin{proof}[Proof of Theorem \ref{th:convex}]
%
%
The implications (i) $\Rightarrow$ (ii) $\Rightarrow$ (iii) hold for any domain in $\R^n$
by Theorem \ref{thm:tauthyp}. 
%
%
The implication (iii) $\Rightarrow$ (iv) is trivial. Indeed, an affine plane $L\subset \R^n$ 
contains arbitrarily big conformal linear discs, and hence for any domain $\Omega\subset\R^n$ 
containing $L$ the minimal pseudodistance $\rho_\Omega(\bp,\bq)$ between any pair of points 
$\bp,\bq\in L$ vanishes.

%
%
The proof of (iv) $\Rightarrow$ (v) 
follows \cite[proof of Proposition 3.5]{BracciSaracco2009}, which pertains to the complex analytic case.
Assume that (iv) holds. We may assume by translation that $\zero\in \Omega$. 
We shall construct $n-1$ linearly independent separating hyperplanes by induction.
Choose a point $\bp\in \Omega^c:=\R^n\setminus \Omega$. Since $\Omega$ is convex, 
there exists a separating hyperplane through $\bp$, i.e., there are a point $\by_1\in\R^n$ 
and a number $a_1\in \R$ such that
$\Omega\subset \{\bx\in\R^n: \bx\cdot\by_1>a_1\}$ and $\bp\cdot\by_1=a_1$.
Assume inductively that for some integer $k\in\{1,\ldots, n-2\}$ there are 
linearly independent separating hyperplanes determined by
the vectors $\by_1, \ldots,\by_k\in\R^n$ and numbers $a_1,\ldots,a_k\in \R$. Their intersection 
\[
	L=\bigcap_{j=1}^k \bigl\{\bx\in\R^n : \bx\cdot\by_j=0\bigr\}
\]
is an $(n-k)$-dimensional linear subspace. 
Since $\Omega$ does not contain any affine $2$-dimensional subspace,
$L\cap \Omega^c$ is nonempty, and we choose a point $\bq\in L\cap \Omega^c$.
By the argument at the beginning of the proof there exists a separating 
hyperplane through $\bq$, i.e.,
there are $\by_{k+1}\in\R^n$ and $a_{k+1}\in \R$
such that $\Omega\subset \{\bx\in\R^n: \bx\cdot\by_{k+1}>a_{k+1}\}$ and $\bq\cdot\by_{k+1}=a_{k+1}$. 
By the construction, the vectors $\by_1, \ldots,\by_{k+1}$ are linearly independent.
This completes the induction and shows that (v) holds.

%
%
In the proof of the implication (v) $\Rightarrow$ (i) we shall need the following lemma, which shows that 
an affine hyperplane $\Sigma\subset\R^n$ is at infinite minimal distance 
from any point of $\R^n\setminus \Sigma$. 

%
%
\begin{lemma}\label{lem:halfspace}
Let $\bx=(x_1,x_2,\ldots,x_n)$ be Euclidean coordinates on $\R^n$, $n\ge 3$, and let
$\H$ denote the half-space
\[
	\H = \bigl\{\bx=(x_1,\bx')\in\R^n: x_1>0\bigr\}.
\] 
If $\bx(t)=(x_1(t),\bx'(t))\in \H$ $(t\in [0,1))$ is a smooth path such that $x_1(t)$ clusters at $0$ or $+\infty$
(i.e., there is a sequence $t_j\in [0,1)$ with $\lim_{j\to\infty} t_j=1$ and 
$\lim_{j\to\infty} x_1(t_j) \in \{0,+\infty\}$), then  
\[
	\int_0^1 g_{\H}(\bx(t),\dot \bx(t)) \, dt =+\infty.
\]
\end{lemma}

\begin{proof}
Denote by $z=x+\imath y$ the complex coordinate on $\C$.
Fix a point $\bx=(x_1,\bx')\in\H$ and a vector $\bv=(v_1,\bv')\in\R^n$ and consider
a harmonic map $f=(f_1,f_2,\ldots,f_n):\D\to \H$ such that $f(0)=\bx$ and
$f_x(0)=r\bv$ for some $r>0$. (We need not assume that $f$ be conformal.)
Hence, $f_1:\D\to(0,+\infty)$ is a harmonic function
with $f_1(0)=x_1$ and $(f_1)_x(0)=rv_1$. It follows that 
\begin{equation}\label{eq:estrv1}
	r|v_1|=|(f_1)_x(0)|\le |\nabla f_1(0)|\le 2|f_1(0)|=2x_1,
\end{equation}
where the second inequality holds by the Schwarz lemma for positive harmonic 
functions. Therefore, 
\[
	 \frac{1}{r} \ge \frac{|v_1|}{2x_1}.
\]
The infimum of the left hand side over all $r>0$ as above equals $g_\H(\bx,\bv)$, so 
\[
	g_\H(\bx,\bv) \ge \frac{|v_1|}{2x_1}.
\]
Given a smooth path $\bx(t)=(x_1(t),\bx'(t))\in \H$ $(t\in [0,1))$, the above shows that 
\[ 
	g_\H(\bx(t),\dot\bx(t)) \ge  \frac{|\dot x_1(t)|}{2x_1(t)},\quad\ t\in [0,1).
\] 
If $x_1(t)\in (0,\infty)$ clusters at $0$ or $+\infty$ as $t\to 1$, then 
\[
	\int_{0}^1  g_\H(\bx(t),\dot\bx(t)) dt \ge
	\frac 12 \int_{0}^1 \frac{|\dot x_1(t)|}{x_1(t)} dt =  
	\frac 12 \int_0^1 |d\log x_1(t)| =+\infty.
\]
This proves the lemma.
\end{proof}

At any boundary point $\bp\in b\Omega$ of a convex domain $\Omega\subset\R^n$ 
there is a supporting affine hyperplane $\Sigma\subset\R^n$ passing through
$\bp$ such that $\Omega$ is contained in one of the half-spaces in $\R^n\setminus \Sigma$.
Hence, by Lemma \ref{lem:halfspace} and the comparison principle \eqref{eq:comparison}, 
any smooth path $\bx(t)\in \Omega$ $(t\in [0,1))$ which clusters at $\bp\in b\Omega$ 
as $t\to 1$ has infinite $g_\Omega$-length. This gives:

%
%
\begin{corollary} \label{cor:locallycomplete}
A convex domain $\Omega\subset\R^n$ $(n\ge 3)$ 
is locally complete hyperbolic at every boundary point $\bp\in b\Omega$.
\end{corollary}

We are now ready to prove the implication (v) $\Rightarrow$ (i).
Assume that $\Omega$ satisfies condition (v).
Up to a translation and rotation on $\R^n$,  there are linearly independent unit vectors 
$\by_1, \ldots,\by_{n-1}\in\R^{n-1}\times\{0\}$ such that, setting
\begin{equation}\label{eq:Hi}
	\H_i= \{\bx\in\R^n: \bx\cdot\by_i>0\},\quad \ i=1,\ldots,n-1, 
\end{equation}
we have that $\Omega\subset \bigcap_{i=1}^{n-1}\H_i$.
Let $\bx(t) =(\bx'(t),x_n(t))$, $t\in[0,1)$, be a  divergent path in $\Omega$. Set 
\begin{equation}\label{eq:xi}
	x_i(t) := \bx(t)\,\cdotp \by_i  = \bx'(t)\,\cdotp \by_i >0, \quad\ i=1,\ldots,n-1,\quad t\in[0,1).
\end{equation}
If $\bx(t) $ clusters at some point $\bp\in b\Omega$ as $t\to 1$, 
then $\bx(t)$ has infinite $g_\Omega$-length by Corollary \ref{cor:locallycomplete}. 
Likewise, if one of the functions $x_i(t)$ for $i=1,\ldots,n-1$ clusters at
$+\infty$, then by Lemma \ref{lem:halfspace} the path $\bx(t)$ has infinite minimal length 
in $\H_i$, and hence also in $\Omega\subset \H_i$.

It remains to consider the case when the functions $x_i(t)$ in \eqref{eq:xi} are bounded, 
\begin{equation}\label{eq:estxi}
	0<\bx(t)\,\cdotp \by_i \le c_1,\quad \ t\in [0,1),\ i=1,\ldots,n-1,
\end{equation}
and the path $\bx(t)\in\Omega$ does not cluster anywhere on $b\Omega$. 
In this case, the last component $x_n(t)\in\R$ of $\bx(t)$ clusters at $\pm \infty$ as $t\to 1$,
and hence $\int_0^1|\dot x_n(t)|dt=+\infty$. Hence, to see that the path $\bx(t)$ has infinite 
$g_\Omega$-length, it suffices to show that 
\begin{equation}\label{eq:lowerbound-xn}
	g_\Omega(\bx(t),\dot\bx(t)) \ge c_2 |\dot x_n(t)| 
\end{equation}
for constant $c_2>0$ depending on $c_1>0$ in \eqref{eq:estxi}
and the vectors $\by_1,\ldots,\by_{n-1}$.

Fix a point $\bx\in\Omega$ satisfying \eqref{eq:estxi} and a unit vector $\bv=(\bv',v_n)\in\R^n$, 
and consider a conformal harmonic map $f=(f_1,f_2,\ldots,f_n):\D\to \Omega$ such that $f(0)=\bx$ and
$f_x(0)=r\bv$ for some $r>0$. Then, $f_y(0)=r\bw=r(\bw',w_n)$ 
where $(\bv,\bw)$ is an orthonormal frame:
\[
	0=\bv\,\cdotp \bw=\bv'\,\cdotp \bw' + v_nw_n,\quad\  |\bv|=|\bw|=1.
\]
From this and the Cauchy--Schwarz inequality it follows that
\[
	v_n^2(1-|\bw'|^2) = v_n^2 w_n^2 
	= |\bv'\,\cdotp \bw'|^2\le |\bv'|^2 |\bw'|^2 = (1-v_n^2)|\bw'|^2,
\]
and hence 
\[
	|v_n| \le |\bw'|\le  c_3 \max_{i=1,\ldots,n-1} |\bw\,\cdotp \by_i|
\]
where the constant $c_3>0$ depends only on the (linearly independent) vectors 
$\by_1,\ldots,\by_{n-1}\in\R^{n-1}\times \{0\}$.  Therefore,
\[
	r |v_n| \le c_3 \max_{i=1,\ldots,n-1} r |\bw\,\cdotp \by_i| \le 
	2c_3 \max_{i=1,\ldots,n-1} \bx\,\cdotp \by_i,
\]
where the second estimate follows from \eqref{eq:estrv1}. (It suffices to apply 
\eqref{eq:estrv1} to the conformal harmonic disc $z\mapsto \tilde f(z)=f(\imath z)$ in each of the 
half-spaces $\H_i$ \eqref{eq:Hi}. Note that $\tilde f_x(0)=f_y(0)=r\bw$.) 
Together with the assumption \eqref{eq:estxi} and setting $c_2=(2c_1c_3)^{-1}$,
this gives
\[
	 \frac{1}{r} \ge \frac{|v_n|}{2c_3 \max_{i=1,\ldots,n-1} \bx\,\cdotp \by_i}
	 \ge \frac{|v_n|}{2c_1c_3} = c_2|v_n|
\]
for any $r>0$ as above. Taking the infimum of the left hand side 
gives $g_\Omega(\bx,\bv) \ge c_2 |v_n|$. Applying this with $\bx=\bx(t)$ and $\bv=\dot\bx(t)$
yields \eqref{eq:lowerbound-xn}. 
This proves that $\Omega$ is complete hyperbolic, so (i) holds.
\end{proof}

%
%
The famous {\em half-space theorem} of Hoffman and Meeks \cite{HoffmanMeeks1990IM}
says that a half-space in $\R^3$ does not contain any non-flat minimal surfaces which are
proper in $\R^3$.  Clearly, properness is a key assumption in their result. 
However, we make the following observation.

\begin{proposition}[Half-space theorem for parabolic minimal surfaces]
\label{prop:halfspace}
Let $\H$ be a half-space in $\R^3$, and let $M$ be the complement of finitely many points
in a compact Riemann surface $R$ without boundary. Then, every conformal minimal surface 
$f:M\to\H$ (possibly branched, not necessarily proper) is flat, with image  
contained in an affine plane parallel to $b\H$.
\end{proposition}

\begin{proof}
Let $\H=\{(x,y,z)\in\R^3:x> 0\}$. The function $u(x)=\frac{1}{x+1}$
is strongly convex on $(0,\infty)$ and takes values in $(0,1)$. Hence, $u$ is MPSH on $\H$,
considered as a function independent of the $y$ and $z$ variables. 
Given a conformal harmonic map $f=(f_1,f_2,f_3):M\to\H$, the function $u\circ f=f_1$ is bounded
subharmonic on $M$, and hence it extends across the punctures to a subharmonic
function on $R$. Since $R$ is compact, the extended function is constant by
the maximum principle, so $f(M)$ lies in a hyperplane $x=const>0$. 
\end{proof}

%
%
%
%
\section{The minimal metric on $\Omega\subset\R^n$ and the Kobayashi metric on $\Omega\times \imath\R^n\subset\C^n$}
\label{sec:min-Kob}

In this section we show that the minimal metric on a domain $\Omega\subset\R^n$ 
$(n\ge 3)$ is in a certain sense bounded from below by the Kobayashi metric on the tube
$\Tcal_\Omega=\Omega\times \imath\R^n\subset\C^n$. In particular, if 
$\Tcal_\Omega$ is Kobayashi (complete) hyperbolic then $\Omega$ is 
(complete) hyperbolic; see Theorem \ref{th:min-Kob}.
 
Fix a point $\bx\in \Omega$ and a $2$-plane $\Lambda\in  G_2(\R^n)$, and let 
$(\bu,\bv)\in\R^n\times \R^n$ be a conformal frame spanning $\Lambda$.
Then, the complex vectors $\bu\pm \imath \bv\in\C^n$ are {\em null vectors}, i.e.,
they belong to the {\em null quadric} 
\[
	\Acal=\bigl\{\bz=(z_1,\ldots,z_n)\in \C^n: z_1^2+z_2^2+\ldots+z_n^2=0\bigr\}.
\]
(See \cite[Section 2.3]{AlarconForstnericLopez2021} for the details.)
Conversely, the real and imaginary part of a null vector $0\ne \xi\in \Acal$ is a conformal frame.
A pair of conformal frames $(\bu,\bv)$ and $(\bu',\bv')$ span the same $2$-plane 
if and only if the corresponding complex vectors satisfy 
\[
	\bu+\imath \bv = c(\bu'\pm \imath\bv')\ \ \text{for some}\ \ c\in \C\setminus \{0\}.
\]
	
Let $f:\D\to\Omega$ be a conformal harmonic disc with $f(0)=\bx$ and $df_0\ne 0$.
Denote by $\zeta=x+\imath y$ the complex coordinate on $\C$. 
The vectors 
\[
	\bu=f_x(0)/|f_x(0)|,\quad \ \bv=f_y(0)/|f_y(0)|=f_y(0)/|f_x(0)|
\]
form an orthonormal frame spanning the 2-plane $\Lambda=df_0(\R^2)$, and
$\|df_0\|=|f_x(0)|=|f_y(0)|$. Let $F=f+\imath g:\D\to \Tcal_\Omega$
be the holomorphic extension of $f$ with $g(0)=0$. Then $F$ is a holomorphic null disc
with $F(0)=\bx$, in the sense that its complex derivative map $F':\D\to \C^n$ has range in the 
null quadric $\Acal$. The Cauchy--Riemann equations imply
\[
	F'(0)=f_x(0)+\imath g_x(0)= f_x(0)-\imath f_y(0)= |f_x(0)| (\bu-\imath \bv). 
\]
Recall that $\Kcal$ denotes the infinitesimal Kobayashi pseudometric.
Since holomorphic null discs in $\Tcal_\Omega$ 
are a subset of the space of all holomorphic discs $\D\to\Tcal_\Omega$,  it follows that
\[
	\Mcal_\Omega(\bx,\Lambda) 
	\ \ge\ \Kcal_{\Tcal_\Omega}\left(\bx,\bu-\imath \bv\right).
\]
From this and the definition of $g_\Omega$ \eqref{eq:g} 
we infer that for any $(\bx,\bu)\in \Omega\times\R^n$ we have that
\begin{equation}\label{eq:gKob}
	g_\Omega(\bx,\bu) \ \ge\ \inf_\bv \Kcal_{\Tcal_\Omega}(\bx,\bu+\imath \bv),
\end{equation}
where the infimum is over all vectors $\bv\in\R^n$ such that $(\bu,\bv)$ is a conformal frame.
(If $\bu\ne 0$ then vectors $\bv$ with this property form an $(n-2)$-sphere in the hyperplane
orthogonal to $\bu$. If $\bu=0$ then both sides equal zero.) 

%
%
\begin{theorem} \label{th:min-Kob}
Let $\Omega$ be a domain in $\R^n$ $(n\ge 3)$, and set 
$\Tcal_\Omega=\Omega\times \imath\R^n\subset\C^n$
\begin{enumerate}[\rm (i)]
\item 
If $\Tcal_\Omega$ is Kobayashi hyperbolic, then $\Omega$ is hyperbolic.
\item 
If $\Tcal_\Omega$ is Kobayashi complete hyperbolic, then $\Omega$ is complete hyperbolic.
\end{enumerate}
\end{theorem}

\begin{proof}
Part (i) is a consequence of  \eqref{eq:gKob} and Royden's characterization 
of hyperbolicity (see Definition \ref{def:Rhyperbolic} 
and Theorem \ref{th:hyperbolic} for minimal hyperbolicity, and 
\cite[Theorem 2, p.\ 133]{Royden1971} for Kobayashi hyperbolicity).

If $\Tcal_\Omega$ is Kobayashi complete hyperbolic, then it is taut and hence pseudoconvex
(see Wu \cite{Wu1967} or Kobayashi \cite[Theorem 5.2.1]{Kobayashi1998}). 
It follows that $\Omega$ is convex \cite[Corollary 2.5.12]{Hormander1990}.
Since $\Omega$ is hyperbolic by part (a), Theorem \ref{th:convex} implies that it is complete
hyperbolic.
\end{proof}

\begin{example}
The implication in Theorem \ref{th:min-Kob} (i) cannot be reversed.
Indeed, let $\Omega\subset\R^3$ be a convex domain which contains an affine line, but it does
not contain any affine $2$-plane. 
Since the tube in $\C^3$ over an affine line in $\R^3$ contains an affine copy of $\C$, 
$\Tcal_\Omega$ is not Kobayashi hyperbolic. On the other hand, since $\Omega$ does not contain 
any affine plane, it is (complete) hyperbolic by Theorem \ref{th:convex}. 
\qed \end{example}

%
%
%
%
\section{A pseudometric defined by minimal log-plurisubharmonic functions}\label{sec:Sibony} 

In this section we introduce a Finsler pseudometric $\Fcal_\Omega$
on any domain $\Omega\subset \R^n$, defined by minimal log-plurisubharmonic functions $\Omega\to [0,1]$; 
see Definition \ref{def:F}. Its main interest is that
it gives a lower bound for the minimal pseudometric $\Mcal_\Omega$ (see Proposition \ref{prop:FM}),
which can often be used to establish hyperbolicity; see e.g.\ Theorem \ref{th:mpsh}. 
We follow ideas from Sibony's paper \cite{Sibony1981}, which pertain to the complex case.

Let $\bx=(x_1,\ldots,x_n)$ be Euclidean coordinates on $\R^n$. Given a domain $\Omega\subset\R^n$
and a function $u:\Omega\to\R$ of class $\Cscr^2$, we denote by $\Hess_u(\bx)=\Hess_u(\bx,\cdotp)$ 
its Hessian form at $\bx\in\Omega$, given by the symmetric matrix $H_u(\bx)=(u_{x_ix_j}(\bx))_{i,j=1}^n$ 
of second order partial derivatives of $u$ at $\bx$. The trace of $\Hess_u$ is the Laplace operator:
\[
	\tr\, \Hess_u = \Delta u = \sum_{i=1}^n u_{x_i x_i} = \sum_{i=1}^n \lambda_i,
\]
where $\lambda_1,\ldots,\lambda_n$ are the eigenvalues of $\Hess_u$.
Given $\bx\in \Omega$ and a $2$-plane $\Lambda\subset \R^n$, let 
\[ 
	\tr_\Lambda \Hess_u(\bx)
\] 
denote the trace of the restriction of $\Hess_u(\bx)$ to $\Lambda$.
Choosing an orthonormal frame $(\bv,\bw)$ for $\Lambda$ and setting 
$\tilde u(x,y) = u(\bx+ x\bv+y\bw)$ for $(x,y)\in\R^2$ near $(0,0)$, it is immediate that 
\begin{equation}\label{eq:Fu2}
	\tr_\Lambda \Hess_u(\bx) = \Delta \tilde u(0,0) = \Delta_{(x,y)} u(\bx+ x\bv+y\bw)|_{x=y=0}.
\end{equation}
%
If $\Lambda$ is a complex line in $\C^n$ then 
$\tr_\Lambda \Hess_u(\bx)$ equals the Levi form $\langle dd^c u|_\bx,\xi\wedge J\xi\rangle$
of $u$ at $\bx$ on any unit vector $\xi\in\Lambda$. (Here, $J$ is the standard almost complex 
structure operator on $\C^n$.)

We recall the following notion from \cite[Definition 8.1.1]{AlarconForstnericLopez2021}. 

%
%
\begin{definition}\label{def:mpsh}
An upper-semicontinuous function $u:\Omega\to[-\infty,+\infty)$ is 
{\em minimal plurisubharmonic}, abbreviated MPSH, if for every affine $2$-plane $L\subset\R^n$
the restriction $u:L\cap\Omega \to [-\infty,+\infty)$ is subharmonic 
in any conformal linear coordinates on $L$.
\end{definition}

Note that every MPSH function on a domain $\Omega$ in $\R^{2n}=\C^n$ 
is also plurisubharmonic in the usual sense of complex analysis. 
Indeed, an upper-semicontinuous function is plurisubharmonic if and only if its 
restriction to every affine complex line is subharmonic (see Klimek \cite{Klimek1991}); 
the latter constitute a small subset of the set of all affine $2$-planes.

By \cite[Proposition 8.1.2]{AlarconForstnericLopez2021}, a function 
$u\in \Cscr^2(\Omega)$ is MPSH if and only if 
\begin{equation}\label{eq:mpsh}
	\tr_\Lambda \Hess_u(\bx)\ge 0 
	\ \  \text{holds for every $(\bx,\Lambda)\in \Omega\times  G_2(\R^n)$},
\end{equation}
and this holds if and only if $\lambda_1(\bx)+\lambda_2(\bx)\ge 0$ for all $\bx\in\Omega$, 
where $\lambda_1(\bx)$ and $\lambda_2(\bx)$ denote the smallest two eigenvalues of 
$\Hess_u(\bx)$. 
We say that $u\in \Cscr^2(\Omega)$ is {\em strongly minimally plurisubharmonic} if strong
inequality holds in \eqref{eq:mpsh}. The key property of MPSH functions pertaining
to minimal surfaces is the following; see \cite[Corollary 8.1.7]{AlarconForstnericLopez2021}.

\begin{proposition}\label{prop:mpsh}
An upper-semicontinuous function $u:\Omega\to[-\infty,+\infty)$ is MPSH if and only if
for each conformal harmonic map $f:M\to \Omega$ from an open conformal surface 
the composition $u\circ f:M\to\R$ is a subharmonic function on $M$.
If in addition $u\in \Cscr^2(\Omega)$ is strongly MPSH and $f$ is an immersion,  
then $u\circ f$ is strongly subharmonic. 
\end{proposition}

For functions $u$ of class $\Cscr^2(\Omega)$, this follows from the following formula,
which holds for every  conformal harmonic map $f:\D\to \Omega$
(see \cite[Lemma 8.1.3]{AlarconForstnericLopez2021}):  
\begin{equation}\label{eq:mainformula}
	\Delta (u\circ f)(z) = \tr_{df_z(\R^2)} \Hess_u(f(z)) \,\cdotp \|df_z\|^2,\quad z\in \D.
\end{equation}
Note that the second order derivative of $f$ does not appear in the formula, and the expression vanishes at 
any critical point of $f$. If $f$ is not conformal harmonic, the formula contains an additional term 
involving the mean curvature vector field of $f$;   
see \cite[Eq.\ (8.6)]{AlarconForstnericLopez2021}.

%
%
\begin{lemma} \label{lem:log}
Let $\bx$ be the Euclidean coordinate on $\R^n$ for $n\ge 3$. 
\begin{enumerate}[\rm (a)]
\item The function $\log|\bx|$ is MPSH on $\R^n$.
\item If $u$ is MPSH and $\chi$ is a convex increasing function defined 
on the range of $u$, then $\chi \circ u$ is MPSH.
In particular, $\bx\mapsto |\bx-\bx_0|^p$ is MPSH on $\R^n$ for every $\bx_0\in \R^n$ and $p>0$.
\item If $u$ is MPSH on $\Omega\subset \R^n$ and $\bp\in\R^n$ then 
the function $\bx\mapsto |\bx-\bp|^2 \E^{u(\bx)}$ and its logarithm are MPSH on $\Omega$.
\end{enumerate}
\end{lemma}

\begin{proof}
The Hessian matrix of $\log|\bx|=\frac 12 \log|\bx|^2$ equals
\[
	H(\bx) = \frac{1}{|\bx|^2} I - \frac{2}{|\bx|^4}\bigl( x_ix_j\bigr)_{i,j=1}^n,
\] 
where $I$ stands for the $n\times n$ identity matrix.
At the point $(p,0,\ldots,0)$ with $p>0$ this equals $p^{-2}\mathrm{diag}(-1,+1,\ldots,+1)$, 
so the sum of the smallest two eigenvalues equals zero. By rotational symmetry 
the same holds at every point of $\R^n\setminus\{0\}$, so $\log|\bx|$ is MPSH.
(See also \cite[Example 2.8]{HarveyLawson2013IUMJ}.)

Part (b) is a consequence of the fact that for any pair of 
functions $g:\D\to\R$ and $\chi:\R\to\R$, 
\begin{equation}\label{eq:Laplacehg}
	\Delta (\chi \circ g) = (\chi'\circ g) \Delta g + (\chi''\circ g) |\nabla g|^2.
\end{equation}
Applying this with $\chi(t)=e^{p t}$ and using that $\log|\bx-\bx_0|$ is MPSH
gives the second statement in (b). Part (c) follows from (a) and (b), noting that if 
$u$ and $v$ are MPSH then so are $u+v$ and $\E^u$.
\end{proof}

\begin{remark}
In the complex case, the function $\log{\sum_{i=1}^k |f_i|^2}$ is plurisubharmonic 
for any collection of holomorphic functions $f_i$ which are not all identically zero
on a given connected complex manifold. Nothing similar holds for MPSH functions.
In fact, it can be seen that for an $n\times n$ matrix $A$, the function $\log |A\bx|$  
is MPSH if and only if $A=rO$ where $r>0$ and $O$ is an orthogonal matrix.
One can get functions which are MPSH in some neighbourhood of $0\in\R^n$ by adding
higher order terms to the argument of $\log$, but such functions do not seem
useful for our subsequent analysis. This again reflects the fact that 
only rigid motions of $\R^n$ preserve the class of conformal minimal surfaces.
\qed\end{remark}

Recall that $ G_2(\R^n)$ denotes the Grassman manifold of $2$-planes in $\R^n$.

%
%
\begin{definition}\label{def:F}
The pseudometric $\Fcal_\Omega:\Omega\times  G_2(\R^n)\to \R_+$ is defined by
\begin{equation}\label{eq:F}
	\Fcal_\Omega(\bx,\Lambda) 
	= \frac12 \sup_u \sqrt{\tr_{\Lambda} \Hess_u(\bx)},
	\quad\ \bx\in \Omega,\ \Lambda\in  G_2(\R^n),
\end{equation}
where the supremum is over all MPSH functions $u:\Omega\to [0,1]$
which are of class $\Cscr^2$ near $\bx$ such that $u(\bx)=0$ and 
$\log u$ is MPSH on $\Omega$. If there are no such functions other than
$u=0$, which may happen if $\Omega$ is unbounded, we set $\Fcal_\Omega(\bx,\Lambda)=0$.
\end{definition}

%
%
\begin{proposition}\label{prop:FM}
On any domain $\Omega\subset\R^n$ $(n\ge 3)$ we have that $\Fcal_\Omega\le \Mcal_\Omega$. 
\end{proposition}

\begin{proof}
Fix $(\bx,\Lambda)\in \Omega\times  G_2(\R^n)$. Let $f\in \CH(\D, \Omega)$ be such that $f(0)=\bx$
and $df_0(\R^2) = \Lambda$. Also, let $u:\Omega\to [0,1]$ be as in 
the definition of the pseudometric $\Fcal_\Omega$ (see Definition \ref{def:F}). 
The function $v=u\circ f:\D\to [0,1]$
is then subharmonic, of class $\Cscr^2$ near the origin, $v(0)=0$, and 
$\log v=\log u\circ f:\D\to [-\infty,0)$ is also subharmonic. 
By Sibony \cite[Proposition 1]{Sibony1981} we have that $\Delta v(0)\le 4$.
(The unique extremal function with $\Delta v(0)= 4$ is $v(x+\imath y)=x^2+y^2$.) 
Hence,  \eqref{eq:mainformula} implies 
$
	\tr_{\Lambda} \Hess_u(\bx) \,\cdotp \|df_0\|^2 = \Delta v(0) \le 4.
$
Equivalently, 
\[
	\frac12 \sqrt{\tr_{\Lambda} \Hess_u(\bx)}  \le \frac{1}{\|df_0\|}.
\]
The supremum of the left hand side over all admissible functions $u$ equals 
$\Fcal_\Omega(\bx,\Lambda)$ (see \eqref{eq:F}), while the infimum of the right hand side 
over all conformal harmonic discs $f$ as above equals $\Mcal_\Omega(\bx,\Lambda)$ 
(see \eqref{eq:MLambda}). This proves Proposition \ref{prop:FM}.
\end{proof}

We now give some applications of the lower bound in Proposition \ref{prop:FM} 
in the spirit of those in \cite{Sibony1981}. 
Pick a smooth increasing function $\theta:[0,\infty)\to [0,1]$ such that 
\begin{equation}\label{eq:theta}
	\theta(t)=t\ \ \text{for}\ 0\le t\le \frac12,\quad\ \ \theta(t)=1\ \ \text{for}\ t\ge 1.
\end{equation}
Let $A>0$ be chosen such that 
\begin{equation}\label{eq:A}
	(\log\theta)'' (t)\ge -A\ \ \text{for all $t\ge \frac12$}.
\end{equation}
Fix a constant $\beta>0$ and set
\begin{equation}\label{eq:h}
	h(\bx) = \log\theta(\beta|\bx|^2),\quad\ \bx\in\R^n.
\end{equation}

%
%
\begin{lemma}\label{lem:estimate1}
If $h$ is given by \eqref{eq:h} and $A$ satisfies \eqref{eq:A}, then 
\begin{equation}\label{eq:loweresth}
	\tr_\Lambda \Hess_h(\bx) \ge -4A\beta 
	\ \ \text{for all $\bx\in \R^n\setminus \{0\}$ and $\Lambda\in  G_2(\R^n)$}.
\end{equation}
\end{lemma}

\begin{proof}
By the definition of $\theta$ and in view of Lemma \ref{lem:log} (a), $h$ is MPSH except perhaps 
on the spherical shell $\frac12 \le \beta|\bx|^2\le 1$. Fix a point $\bx$ in this shell and a 2-plane 
$\Lambda\subset\R^n$. Let $(\bu,\bv)$ be an orthonormal frame for $\Lambda$. 
Consider the function $g:\R^2\to \R_+$ given by
\[
		g(x,y) =\beta|\bx+x\bu+y\bv|^2 
	         = \beta \left( |\bx|^2 + 2x (\bx\cdotp\bu) + 2y (\bx\,\cdotp \bv) + x^2+y^2\right).
\]
We have that $\Delta g = 4\beta$, $g_x(0,0)=2\beta\, \bx\cdotp\bu$, 
$g_y(0,0)=2\beta\, \bx\cdotp\bv$, and hence (since $\beta|\bx|^2\le 1$)
\begin{equation}\label{eq:nablag}
		|\nabla g(0,0)|^2 = 4\beta^2 \left(|\bx\cdotp\bu|^2+|\bx\cdotp\bv|^2\right) 
		\le 4\beta^2 |\bx|^2 \le 4\beta.
\end{equation}
Applying the formula \eqref{eq:Laplacehg} with the increasing function $\chi=\log \theta$ gives
\[
	\tr_\Lambda \Hess_h(\bx) = 
	\Delta (\log\theta\circ g)(0,0) \ge \theta''(\beta|\bx|^2) |\nabla g(0,0)|^2 \ge -4A\beta.
\]
This is the estimate \eqref{eq:loweresth}.
\end{proof}

%
%
\begin{lemma}\label{lem:lowerbound}
Let $\Omega$ be a domain in $\R^n$ $(n\ge 3)$, $\bp\in \R^n$, and $r>0$. 
Assume that $u:\Omega \to [-\infty,0)$ is an MPSH function of class
$\Cscr^2$ such that for some $c>0$ we have 
\begin{equation}\label{eq:estu}
	\tr_\Lambda \Hess_u(\bx)\ge c>0
	\ \ \text{for all $\bx\in \Omega \cap \B(\bp,2r)$ and $\Lambda\in G_2(\R^n)$}. 
\end{equation}
Let the constant $A>0$ satisfy \eqref{eq:A}. Then we have that
\begin{equation}\label{eq:estg1}
	g_\Omega(\bx,\bu) \ge \frac{1}{r}\, \E^{2Au(\bx)/cr^2} |\bu| 
	\ \ \ \text{for all $\bx\in \B(\bp,r)\cap\Omega$ and $\bu\in\R^n$}. 
\end{equation}
\end{lemma}

\begin{proof}
Let the function $\theta$ 
be as in \eqref{eq:theta}. 
Fix a point $\bx\in \B(\bp,r)\cap\Omega$.
Given a constant $\lambda>0$, we consider the smooth function 
\[ 
	\Psi(\by) = \theta\left( r^{-2}|\by-\bx|^2\right) \E^{\lambda u(\by)},\quad\ \by\in \Omega.
\] 
Note that $0\le \Psi\le 1$, $\Psi(\bx)=0$, and $\log \Psi$ is MPSH except perhaps on the
spherical shell 
\[	
	S_r=\left\{\by\in \R^n:  \frac{r}{\sqrt 2} \le |\by-\bx| \le r \right\} \subset \B(\bp,2r)
\]
intersected with $\Omega$.
On $S_r \cap\Omega$ we have by Lemma \ref{lem:estimate1} and the assumption \eqref{eq:estu} that 
\begin{equation}\label{eq:estlogPsi}
	\tr_\Lambda \Hess_{\log \Psi} \ge -\frac{4A}{r^2} + c \lambda,
	\quad \ \Lambda\in  G_2(\R^n).
\end{equation}
Choosing $\lambda=\frac{4A}{cr^2}$ we ensure that $\log \Psi$ is MPSH on $\Omega$.
For this value of $\lambda$, the function $\Psi$ is a candidate
for computing $\Fcal_\Omega(\bx,\Lambda)$ \eqref{eq:F}. A calculation gives 
for all $\Lambda\in  G_2(\R^n)$ that 
\begin{equation}\label{eq:trHessPsi}
	\tr_\Lambda \Hess_{\Psi}(\bx) 
	= \frac{4}{r^2} \, \E^{\lambda u(\bx)} 
	= \frac{4}{r^2} \, \E^{4Au(\bx)/cr^2}.
\end{equation}
Taking into account \eqref{eq:F} and Proposition \ref{prop:FM} it follows that
\[
	\Mcal_\Omega(\bx,\Lambda) \ge \Fcal_\Omega(\bx,\Lambda) 
	\ge \frac12 \sqrt{\tr_\Lambda \Hess_{\Psi}(\bx)}
	= \frac{1}{r}\, \E^{2Au(\bx)/cr^2}.
\]
From this and the property \eqref{eq:gM} of the metric $g_\Omega$ we clearly get \eqref{eq:estg1}.
\end{proof}

%
%
%
%
The following result is an analogue of \cite[Theorem 3, Corollary 5, and Proposition 6]{Sibony1981}.

\begin{theorem}\label{th:mpsh}
Let $\Omega$ be a domain in $\R^n$ for $n\ge 3$.
\begin{enumerate}[\rm (a)] 
\item If there is a negative MPSH function on $\Omega$ which is smooth strongly MPSH on a neighbourhood 
of a point $\bp\in \Omega$, then $\Omega$ is hyperbolic at $\bp$ (see Definition \ref{def:Rhyperbolic}).
\item If $u$ is a negative $\Cscr^2$ function on a domain $\Omega\subset \R^n$ such that 
\begin{equation}\label{eq:estu2}
	\tr_\Lambda \Hess_u(\bx)\ge c>0\ \ \text{for all $\bx\in \Omega$ and $\Lambda\in G_2(\R^n)$}
\end{equation}
and  $A>0$ is as in \eqref{eq:A}, then the minimal metric $g_\Omega$ satisfies the estimate 
\begin{equation}\label{eq:estg}
	g_\Omega(\bx, \bv) \ge \sqrt{\frac{c}{4A\E }} \frac{|\bv|}{\sqrt{|u(\bx)|}}
	\ \ \text{for all $\bx\in \Omega$ and $\bv\in\R^n$}. 
\end{equation}
\item A domain with a bounded strongly MPSH function is hyperbolic.
\item If $\Omega$ admits a continuous MPSH exhaustion function $u:\Omega\to [-\infty,0)$
(such a domain is called hyperconvex, see Definition \ref{def:hyperconvex}),
then $\Omega$ is hyperbolic.
\end{enumerate}
\end{theorem}

\begin{proof}
Part (a) is an immediate consequence of Lemma \ref{lem:lowerbound} and Theorem \ref{th:hyperbolic}.

Part (b) follows from \eqref{eq:estg1} by choosing $\bx=\bp\in\Omega$ and $r^2= 4A|u(\bx)|/c$.

If the assumption in (c) holds then the domain is hyperbolic at every point by (a), and hence it is hyperbolic
by the implication (iii)\ $\Longrightarrow$\ (iv) in Theorem \ref{th:hyperbolic}.

To prove (d), assume first that $u:\Omega\to (-\infty,0)$ is an MPSH exhaustion function of class $\Cscr^2$.
Pick a number $t_0<0$. The set $K=\{\bx \in\Omega: u(\bx)\le t_0\}$ is compact since $u$ is an exhaustion. 
Choose a relatively compact neighbourhood $V\Subset \Omega$ of $K$ and set 
\[
	\mu=\inf_{\bx\in bV} u(\bx) > t_0,\qquad c=\max_{\bx\in bV}\left(u(\bx)+|\bx|^2\right).
\]
Pick an increasing convex function $h:\R\to\R$ such that 
\[
	h(t)=t\ \ \text{for $t\le t_0$},\quad\  h(\mu)>c.
\]
Define the function $\psi:\Omega\to\R$ by
\[	
	\psi(\bx)= \begin{cases} 
			\max\{u(\bx)+|\bx|^2,h(u(\bx))\}, & \text{if $\bx\in V$}; \\
			h(u(\bx)), 					   & \text{if $\bx\in \Omega\setminus V$}.
			\end{cases}
\]
Then, $\psi$ is a bounded MPSH exhaustion function on $\Omega$ such that $\psi(\bx)=u(\bx)+|\bx|^2$
for $\bx$ in a neighbourhood of $K$, so it is strongly MPSH there. It follows that $\Omega$ is
hyperbolic at every point of $K$. Letting $t_0\to 1$, the compact set $K=\{u\le t_0\}$ increases to $\Omega$,
so $\Omega$ is hyperbolic by Theorem \ref{th:hyperbolic} (the equivalence 
(iii) $\Leftrightarrow$ (iv)). This proves (d) if $u$ is smooth.

Suppose now that $u:\Omega\to [-\infty,0)$ is a continuous MPSH exhaustion function.
Choose $t_0<0$ close enough to $0$ such that the compact set $K=\{u\le t_0\}$ 
contains the level set $\{u=-\infty\}$ in its interior. 
Taking the convolution of $u$ with a smooth radially symmetric approximate identity
gives a smooth MPSH function $v\ge u$ on a neighbourhood of $K$.  
(See \cite[Proposition 8.1.6]{AlarconForstnericLopez2021}. Note that 
smoothing also applies if $u$ assumes the value $-\infty$, 
the reason being that an MPSH function is subharmonic in the usual
sense, hence locally integrable.) The same procedure as in the special
case above, using $v(\bx)+|\bx|^2$ as the first item under $\max$ in the definition of $\psi$, 
yields another bounded MPSH exhaustion $\tilde u$ on $\Omega$ which is strongly MPSH in a 
neighbourhood of $K$, so $\Omega$ is hyperbolic there. Letting $t_0\to 0$ completes the proof.
\end{proof}

Here are examples of unbounded domains satisfying conditions in Theorem \ref{th:mpsh} (c), (d). 

%
%
\begin{example}\label{ex:Omegaab}
Consider the function $u(x,y,z)=x^2 + y^2 - az^2$ on $\R^3$. 
Its Hessian matrix has eigenvalues $2,2,-2a$, so $u$ is strongly MPSH and satisfies  
condition \eqref{eq:estu2} when $a<1$. For such $a$ and any $b\in\R$ the unbounded 
strongly minimally convex domain
\[
	\Omega_{a,b}=\{u<0\}=\{(x,y,z)\in\R^3: x^2 + y^2 < az^2 + b\}
\]
is hyperbolic by Theorem \ref{th:mpsh} (c).  
Part (b) of the same theorem shows that 
$g_\Omega(\bx,\bv)\ge const \frac{|\bv|}{\sqrt{|u(x,y,z)|}}$ 
for all $(x,y,z)\in\Omega$ and $\bv\in\R^n$.  Using this estimate and the fact that 
$\Omega_{a,b}$ is strongly minimally convex, it can be seen that $\Omega_{a,b}$ 
is in fact complete hyperbolic (see Corollary \ref{cor:Omegaab}).
\qed\end{example}


%
%
\begin{example}\label{ex:hyperconvex}
Let $h:\R\to (-\infty,0)$ be a smooth function such that $\lim_{|x|\to\infty}h(x)=0$ and
$|h''(x)|\le c<+\infty$ for all $x\in\R$. An example is $h(x)=-1/(1+x^2)$. Consider the function
\[
	u(x,y,z)=h(x)+a(y^2+z^2),\quad \ (x,y,z)\in\R^3
\]
for some $a>0$. The Hessian matrix of $u$ at $(x,y,z)$ is the diagonal matrix with the eigenvalues 
$h''(x),2a,2a$. Since $|h''(x)|\le c$, the sum of any two eigenvalues is nonnegative
if $a\ge c/2$. For such $a$, $u$ is a negative MPSH exhaustion function on 
the unbounded domain 
\[
	\Omega=\{u<0\}=\{(x,y,z)\in\R^3: y^2+z^2 < -h(x)/a\}
\]
containing the line $\R\times \{(0,0)\}$.
Hence, $\Omega$ is hyperbolic by Theorem \ref{th:mpsh} (d).
\qed\end{example}

%
%
%
%
\section{Localization theorems for the minimal pseudometric}\label{sec:localization}

Given domains $\Omega_0\subset \Omega\subset \R^n$, their minimal pseudometrics 
satisfy $g_{\Omega}(\bx,\bv)\le g_{\Omega_0}(\bx,\bv)$ for every $\bx\in \Omega_0$
and $\bv\in \R^n$; see \eqref{eq:comparison}. 
Assuming that a point $\bp$ is contained in the relative interior of 
$b\Omega\cap b\Omega_0$, it is often possible to obtain the inverse inequality 
up to a positive multiplicative constant. 
The motivation is that it is easier to estimate the metric from below
on small subdomains. There are several known estimates of this type for 
biholomorphically invariant metrics; see the papers 
\cite{Sibony1981,ForstnericRosay1987,Berteloot1994,Gaussier1999,Nikolov2002,IvashkovichRosay2004,DiederichFornaessWold2014,FornaessNikolov2021}, among others. 
We shall prove the following result concerning localization of the minimal pseudometric.

%
%
\begin{theorem}\label{th:loc1}
Let $\Omega$ be a domain in $\R^n$, $n\ge 3$. Assume that for some point 
$\bp\in b\Omega$ there exist a neighbourhood $U\subset \R^n$ of $\bp$ 
and a continuous function $u:\overline \Omega\cap U \to (-\infty, 0]$ such that $u(\bp)=0$, 
$u<0$ on $\overline\Omega\cap U \setminus \{\bp\}$, $u$ is MPSH on $\Omega\cap U$
(see Definition \ref{def:mpsh}), and
\begin{equation}\label{eq:assloc1}
	|u(\bx)| \le c_0 |\bx-\bp|,\quad\ \bx\in \Omega\cap U
\end{equation}
holds for some $c_0>0$. Given $r_0>0$, there is a constant $c>0$ such that 
\begin{equation}\label{eq:localestimateg}
	g_\Omega(\bx,\bv) \ge (1-c|\bx -\bp|)\, g_{\Omega\cap \B(\bp,r_0)}(\bx,\bv),\quad\
	 \bx\in  \Omega\cap \B(\bp,r_0),\ \bv\in\R^n.
\end{equation}
\end{theorem}

This result is close in spirit to \cite[Theorem 2.1]{ForstnericRosay1987}; see also Theorem \ref{th:loc2}.

%
%
\begin{remark}\label{rem:localization}
(A) The estimate \eqref{eq:localestimateg} 
says in particular that $\Omega$ is hyperbolic at every point close enough to $\bp$. 
An analogous result for the Kobayashi metric is \cite[Theorem 2.1]{ForstnericRosay1987}. 
Our proof gives a stronger conclusion under weaker hypotheses also in that situation: 
the domain $\Omega\subset \C^n$ need not be bounded, the local peak function $u$ 
may be plurisubharmonic instead of holomorphic, and there is no condition 
on the upper bound of $|\bx-\bp|$ in terms of $|u(\bx)|$.

(B) Our proof of \eqref{eq:localestimateg} can be adjusted to the situation
studied by Ivashkovich and Rosay in \cite[Sect.\ 2a]{IvashkovichRosay2004}. 
The upshot is that, under their hypotheses, the quotient of the infinitesimal Kobayashi metrics 
of a domain $\Omega$ and the subdomain $\Omega\cap \,\B(\bp,r_0)$ 
in an almost complex manifold $X$ approaches $1$ linearly in terms of $\dist(\bx,\bp)$.
(Ivashkovich and Rosay did not provide a quantitative estimate of the rate of convergence.)  
\qed \end{remark}

In the proof of Theorem \ref{th:loc1} we shall need the following lemma.

%
%
\begin{lemma} \label{lem:Harnack}
Let $K$ be a nonempty compact set in the circle $\T=b\D =\{|z|=1\}$, 
and let $v(z)$ be the positive harmonic function on $\D$ whose boundary values 
on $\T$ agree with the characteristic function of $K$. 
Let $|K|$ denote the normalized Lebesgue measure of $K$.
Then, for any $\mu \in (0,1)$ we have that
\begin{equation}\label{eq:lessthanc}
	v(z)\le \mu \ \ \text{for}\ \ |z|\le 1-\frac{2}{\mu}|K|. 
\end{equation}
\end{lemma}

\begin{proof}
A positive harmonic function $v$ on $\D$ satisfies Harnack's inequality
\[
	\frac{1-r}{1+r} v(0) \le v(r \E^{\imath t}) \le \frac{1+r}{1-r} v(0),\quad 0\le r<1,\ t\in\R.
\]
In our case we have that $v(0)=|K|$, so 
\[
	v(r \E^{\imath t}) \le  \frac{1+r}{1-r} v(0) \le \frac{2}{1-r} |K| \le \mu
\]
where the last inequality holds for $r\le 1-\frac{2}{\mu}|K|$. This gives \eqref{eq:lessthanc}.
\end{proof}

%
%
\begin{proof}[Proof of Theorem \ref{th:loc1}]
The estimate \eqref{eq:localestimateg} clearly follows from the following.

\begin{lemma} \label{lem:localization}
{\rm (Assumptions as in Theorem \ref{th:loc1}.)} 
Given $r_0>0$, there is $c=c(r_0)>0$ such that for every conformal harmonic disc $f:\D\to \Omega$ 
we have that 
\begin{equation}\label{eq:estloc1}
	|f(z)-\bp| \le r_0\ \ \text{for all}\ \ |z|\le 1-c|f(0)-\bp|. 
\end{equation}
\end{lemma}

Setting $\epsilon = |f(0)-\bp|$, condition \eqref{eq:estloc1} says that $f$ maps the 
disc of radius $r=1-c\epsilon$ into the ball $\B(\bp,r_0)$. It follows that if 
$(f_k)_{k\in \N}\subset \CH(\D,\Omega)$ is such that $\lim_{k\to\infty}f_k(0)=\bp$, 
then the sequence $f_k$ converges to $\bp$ uniformly on compacts in $\D$.

It remains to prove the lemma. 
Replacing the function $u$ in Theorem \ref{th:loc1} 
by $\max\{u,-c_1\}$ for a suitably chosen $c_1>0$, 
we may assume that $u$ is defined on $\overline \Omega$ and MPSH 
on $\Omega$, and \eqref{eq:assloc1} holds for all $\bx\in \Omega$.

For convenience of exposition, we first consider the case when $\Omega$ is bounded.
We may assume that $\bp=\zero$ and 
$\Omega$ is contained in the unit ball $\B^n$ centred at the origin.
It suffices to prove \eqref{eq:estloc1} for small numbers $r_0\in (0,1)$ and for discs $f\in \CH(\D,\Omega)$ 
with centres $f(0)$ close to $\zero$. The assumptions on $u$ imply that
\begin{equation}\label{eq:epsilon0}
	-\epsilon_0 := \sup \bigl\{u(\bx): \bx\in \overline\Omega,\ |\bx|\ge r_0^2\bigr\} <0.
\end{equation}
Fix $f\in \CH(\D,\Omega)$ and set $\epsilon = -u(f(0))>0$. Since $\bp=\zero$, \eqref{eq:assloc1} gives
\begin{equation}\label{eq:est2loc1}
	\epsilon \le c_0 |f(0)|. 
\end{equation}
Since $f:\D\to \R^n$ is a bounded harmonic map, it has a nontangential
limit at almost every point of the circle $\T=b\D$, the boundary map is of class $L^\infty(\T)$,
and $f$ is the Poisson integral of its boundary function. Consider the measurable set
\begin{equation}\label{eq:K}
	K=\bigl\{t\in [0,2\pi]: |f(\E^{\imath t})|\ge r_0^2\big\}.
\end{equation}
The choice of $\epsilon_0$ in \eqref{eq:epsilon0} ensures that  
\[
	t \in K\ \Longrightarrow \ u\circ f(\E^{\imath t}) \le -\epsilon_0.
\]
Since the function $u\circ f$ is negative subharmonic on $\D$ (see Proposition \ref{prop:mpsh}), it follows that 
\[
	-\epsilon = u(f(0)) \le \int_0^{2\pi} u\circ f(\E^{\imath t}) \frac{dt}{2\pi} 
	\le \int_K  u\circ f(\E^{\imath t}) \frac{dt}{2\pi}
	\le -\epsilon_0 |K|.
\]
From this and \eqref{eq:est2loc1} we obtain the estimate
\begin{equation}\label{eq:estK}
	|K| \le \frac{\epsilon}{\epsilon_0} \le \frac{c_0 |f(0)|}{\epsilon_0}.
\end{equation}

By the assumption made at the beginning of the proof, the function $\log|f|<0$ is negative
on $\D$, and it is subharmonic on $\D$ by Lemma \ref{lem:log} (a). 
Hence, $\log|f(\E^{\imath t})|\le 0$ for all $t$, and  
\[
	\log|f(\E^{\imath t})| < \log r_0^2=2\log r_0<0\ \ \text{for}\ \ t \in [0,2\pi]\setminus K
\]
in view of the definition of the set $K$ \eqref{eq:K}. 
Let $\lambda:\D\to (-\infty,0)$ be the harmonic function whose 
a.e.\ boundary values equal $0$ on $K$ and $2\log r_0$ on $\T\setminus K$.
Then, $\log|f|\le \lambda$ on $\D$ by subharmonicity of $\log|f|$. Note that
$
	\lambda=2\log r_0 + 2|\log r_0| v_K,
$
where $v_K:\D\to (0,1]$ is the Poisson extension of the characteristic function $\chi_K$ of $K$
(the harmonic measure of $K$).
By Lemma \ref{lem:Harnack}, applied to $v_K$ with $\mu=1/2$, we infer that 
$v_K\le 1/2$ on the disc $|z|\le 1-4|K|$. By \eqref{eq:estK}, this contains the disc
$|z|\le 1-c|f(0)|$ with $c=4c_0/\epsilon_0$. For such $z$ we get
\[
	\log|f(z)| \le \lambda(z) \le 2\log r_0 +  |\log r_0| = \log r_0
\]
and hence $|f(z)| \le r_0$. This proves \eqref{eq:estloc1} (and hence the theorem) 
if $\Omega$ is bounded.

If $\Omega$ is unbounded, we replace the function $\log|\bx|$ in the last step of the proof 
by a negative MPSH function $\zeta:\Omega\to (-\infty,0)$ which equals $\log|\bx|+O(1)$ on 
$\Omega\cap U'$ 
for some neighbourhood $U'\subset \R^n$ of the origin, and it agrees with $cu$ for some $c>0$
outside a larger neighbourhood of the origin. The construction of $\zeta$ is elementary; 
see the construction of the function $\lambda^\sharp$
in \cite[pp.\ 2410--2411]{IvashkovichRosay2004}. Assuming that the conformal 
harmonic disc $f:\D\to\Omega$ is bounded, the proof may be completed as before by considering 
the negative subharmonic function $\zeta\circ f:\D\to (-\infty,0)$. The constants in the proof need 
a minor adjustment depending only on the 
size of the bounded difference $|\zeta(\bx)-\log|\bx||$ on $\bx\in \Omega\cap U'$, 
but not on the disc $f$. Hence, the conclusion also holds if $f$ is unbounded, replacing it by 
the bounded discs $\D\ni z\mapsto f(r'z)$ for $r'<1$ and letting $r'\to 1$.
\end{proof}
%
%

One is often interested in a lower estimate on $g_\Omega(\bx,\cdotp)$ in terms
of the boundary distance $\dist(\bx,b\Omega)$.
Such an estimate can be obtained from the proof of Theorem \ref{th:loc1},  
provided that there is a family of local negative MPSH peaking functions as in 
the following theorem. 

%
%
\begin{theorem}\label{th:loc2}
Let $\Omega_1\subset \Omega$ be domains in $\R^n$ $(n\ge 3)$ with $\overline \Omega_1$ compact, 
and let $\omega \subset b\Omega \cap b\Omega_1$ be such that for some $\delta>0$ we have
$\B(\bp,\delta)\cap \Omega \subset \Omega_1$ for all $\bp\in \omega$. 
Assume that for every $\bp\in\omega$ there exists a continuous
function $u_\bp:\overline \Omega_1 \to (-\infty, 0]$, which is MPSH on $\Omega_1$, 
such that $u_\bp(\bp)=0$, $u_\bp<0$ on $\overline\Omega_1 \setminus \{\bp\}$, and 
\begin{equation}\label{eq:assloc2}
	c_0|u_\bp(\bx)| \le |\bx-\bp| \le \phi(|u_\bp(\bx)|),\quad\ \bx\in \Omega_1,\ \bp\in\omega
\end{equation}
for some $c_0>0$ and a continuous increasing function $\phi:\R_+\to\R_+$
with $\phi(0)=0$. (The constant $c_0$ and the function $\phi$ are independent of $\bp\in\omega$.)
Then there is $c>0$ such that
\begin{equation}\label{eq:localestimateg2}
	g_\Omega(\bx,\bv) \ge (1-c\,\dist(\bx,\omega))\, g_{\Omega_1}(\bx,\bv),\quad\
	 \bx\in  \Omega_1,\ \bv\in\R^n.
\end{equation}
In particular, $\Omega$ is hyperbolic at every point close enough to $\omega$.
\end{theorem} 
 
Theorem \ref{th:loc2} follows from the proof of Theorem \ref{th:loc1} by observing that the second inequality
in \eqref{eq:assloc2} ensures that for any given $r_0>0$, the number $\epsilon_0>0$
in \eqref{eq:epsilon0} (with $u$ replaced by $u_\bp$) can be chosen independent of $\bp\in\omega$.
In fact, the smallest number $\epsilon_0>0$ such that $\phi(\epsilon_0)=r_0^2$
satisfies \eqref{eq:epsilon0}. This implies that the constant $c=c(r_0)>0$ in 
\eqref{eq:estloc1} can be chosen independent of $\bp\in\omega$.
We leave further details to the reader.

%
%
%
%
\section{Bounded strongly minimally convex domains are complete hyperbolic}\label{sec:mpsh}

A domain $\Omega\subset \R^n$ $(n\ge 3)$ with $\Cscr^2$ boundary is 
{\em strongly minimally convex} if it admits a $\Cscr^2$ defining function $u:U\to \R$ on a 
neighbourhood $U\subset\R^n$ of $\overline \Omega$
(i.e., $\Omega=\{\bx\in U:u(x)<0\}$ and $du\ne 0$ on $b\Omega=\{u=0\}$)
which is strongly MPSH on a neighbourhood $V\subset U$ of $b\Omega$.
If $\overline \Omega$ is compact, then $u$ can be chosen strongly MPSH on 
a neighbourhood of $\overline \Omega$.
(See \cite[Definition 8.1.18 and Lemma 8.1.19]{AlarconForstnericLopez2021}.) 
Since $|u(\bx)|$ is comparable to $\dist(\bx,b\Omega)$ for $\bx$ sufficiently close 
to $b\Omega$, the following is a corollary to Theorem \ref{th:mpsh} (b).

%
%
\begin{corollary}\label{cor:spsc} 
For every bounded strongly minimally convex domain $\Omega$ in $\R^n,\ n\ge 3,$ 
there is a constant $C>0$ such that 
\begin{equation}\label{eq:estg2}
	g_\Omega(\bx, \bv) \ge C \frac{|\bv|}{\sqrt{\dist(\bx,b\Omega)}}
	\ \ \ \text{for all $\bx\in\Omega$ and $\bv\in\R^n$}.
\end{equation}
\end{corollary}

From the formula \eqref{eq:CK} for the metric $g_{\B^n}$ on the unit ball of $\R^n$
we see that, except for the size of the constant $C$, the asymptotic 
estimate \eqref{eq:estg2} is the best possible for all tangent vectors $\bv$. 
However, if $\bv$ is normal to the boundary $b\Omega$ 
at the closest point $\bp\in b\Omega$ to $\bx$ (such $\bp$ is unique if $\bx$ is close enough to 
$b\Omega$), we expect a lower bound 
\begin{equation}\label{eq:radialestimate}
	g_\Omega(\bx, \bv) \ge C \frac{|\bv|}{\dist(\bx,b\Omega)}
\end{equation}
for some $C>0$. The estimate \eqref{eq:radialestimate}
implies that $\Omega$ is complete hyperbolic. 
This follows from \cite[Lemma 1.1, p.\ 2396]{IvashkovichRosay2004} 
applied to conformal harmonic discs. 

For the Kobayashi metric on strongly pseudoconvex domains in $\C^n$,
the estimate \eqref{eq:radialestimate} for radial vectors $\bv$ was obtained by 
Graham \cite{Graham1975}; see also \cite[Proposition 7]{Sibony1981}. 
Another proof under the weaker assumption
that $b\Omega$ is strongly pseudoconvex near a point $\bp\in b\Omega$ 
and $\bx\in\Omega$ is close to $\bp$ was given by Forstneri\v c and Rosay 
\cite[Theorem 2.1]{ForstnericRosay1987}. 
We could not adapt those proofs to the present situation. 
Instead, we obtain the estimate \eqref{eq:radialestimate} 
by following \cite[proof of Theorem 1]{IvashkovichRosay2004}, 
thereby obtaining the following result.

%
%
\begin{theorem}\label{th:completehyperbolic}
Let $\Omega$ be a (not necessarily bounded) domain in $\R^n$, $n\ge 3$.  
If the boundary $b\Omega$ is strongly minimally convex at a point $\bp\in b\Omega$, then 
$\bp$ is at infinite minimal distance from $\Omega$. In particular, every bounded strongly 
minimally convex domain is complete hyperbolic.
\end{theorem}

%
%
\begin{remark}\label{rem:complete}
A domain $\Omega\subset \R^n$ with $\Cscr^2$ boundary 
is strongly minimally convex if and only if at every point $\bp\in b\Omega$
the interior principal curvatures $\nu_1\le \nu_2\le \cdots\le \nu_{n-1}$ of $b\Omega$ 
satisfy $\nu_1+\nu_2>0$ (see \cite[Theorem 8.1.13]{AlarconForstnericLopez2021}). 
If at some point $\bp\in b\Omega$ we have $\nu_1+\nu_2<0$, 
then $\bp$ is at finite minimal distance from $\Omega$.
Indeed, there is an embedded conformal minimal disc in $\Omega\cup\{\bp\}$
centred at $\bp$ \cite[Lemma 3.13]{HarveyLawson2013IUMJ}, 
so Example \ref{ex:notcompletehyperbolic} applies.
The following consequence seems worthwhile recording.
We do not know whether the converse holds; see Problem C in Section \ref{sec:problems}.
\end{remark}

\begin{corollary}
If $M$ is a smooth embedded surface in $\R^3$ such that the minimal distance
to any point of $M$ from both sides is infinite, then $M$ has vanishing mean curvature, 
and hence is a minimal surface.
\end{corollary}

%
%
To prove Theorem \ref{th:completehyperbolic} we shall need a couple of lemmas. 
The first one is an analogue of \cite[Lemma 2.2]{IvashkovichRosay2004}. 

\begin{lemma}\label{lem:distanceestimate}
If the boundary of the domain $\Omega\subset\R^n,\ n\ge 3,$ is strongly minimally convex
near $\bp\in b\Omega$, then for any $r\in (0,1)$ there are constants $\delta>0$ and $C>0$
such that every $f\in \CH(\D,\Omega)$ with $|f(0)-\bp|<\delta$ satisfies
\begin{equation}\label{eq:estdisctance}
	|f(z)-f(0)| \le C\sqrt{\dist(f(0),b\Omega)}\ \ \ \text{for all $|z|\le r$}.
\end{equation}
\end{lemma}

\begin{proof}
Pick a ball $U\subset \R^n$ around $\bp$ such that $\Omega$ admits 
a strongly MPSH defining function $u$ in a neighbourhood of $\overline {\Omega\cap U}$.
For $\epsilon>0$ small, the function $\rho_\bq(\bx)= u(\bx)-\epsilon |\bx-\bq|^2$
is strongly MPSH on $\overline {\Omega\cap U}$ for every $\bq\in b\Omega\cap U$.
Up to shrinking $U$ around $\bp$ if necessary, there are constants $a,b>0$ such that for 
every $\bq$ as above we have
\begin{equation}\label{eq:estAB}
	-b |\bx-\bq|\le \rho_\bq(\bx) \le -a |\bx-\bq|^2,\quad \bx\in \overline \Omega\cap U.
\end{equation}
Fix $r\in [0,1)$ and let $r'=(1+r)/2$. There is a $\delta >0$ 
such that for every $f\in \CH(\D,\Omega)$ with $\dist(f(0),\bp)<\delta$ we have that 
$f(r'\D)\subset U\cap \Omega$ by Lemma \ref{lem:localization}.
Hence, by rescaling the disc we may assume that this holds for $r'=1$.
(This will change the final constant $C$ by a positive factor.) 
Let $\bq\in b\Omega\cap U$ be the nearest point to $f(0)$ in $b\Omega$. 
By Harnack's theorem (on the growth on positive harmonic functions in the disc) and
subharmonicity of the function $z\mapsto |f(z)-\bq|^2$ there is a constant $c=c(r)>0$ such that
\[
	|f(z)-\bq|^2 \le c \int_0^{2\pi} |f(\E^{\imath t})-\bq|^2\frac{dt}{2\pi} \quad
	\text{for all $|z|\le r$}.
\]
Since $\rho_\bq \circ f$ is subharmonic on $\D$, it follows that for all $|z|\le r$ we have 
\[
	b |f(0)-\bq| 
	\ge -\rho_\bq(f(0)) \ge - \int_0^{2\pi} \rho(f(\E^{\imath t})) \frac{dt}{2\pi}
	\ge a \int_0^{2\pi} |f(\E^{\imath t}) - \bq|^2\frac{dt}{2\pi} 
	\ge \frac{a}{c} |f(z)-\bq|^2.
\]
(The first and the third inequality hold by \eqref{eq:estAB}, the second one is due to subharmonicity
of $\rho_\bq \circ f$, and the last one uses the previous estimate.) Setting $C=bc/a>0$ we obtain 
\[
	|f(z)-\bq|^2 \le  C |f(0)-\bq| = C\, \dist(f(0),b\Omega)\quad \text{for all $|z|\le r$}.
\]
Since $|f(z)-f(0)|^2\le 2( |f(z)-\bq|^2 + |f(0)-\bq|^2)=  2( |f(z)-\bq|^2 + \dist(f(0),b\Omega)^2)$,
we obtain the desired estimate \eqref{eq:estdisctance} with a different constant $C$ provided 
$\dist(f(0),b\Omega)$ is small enough, which holds if $|f(0)-\bp|$ is small enough. 
\end{proof}

The next lemma provides a Cauchy estimate for harmonic discs in $\R^n$; we omit the proof.

%
%
\begin{lemma}\label{lem:Cauchyestimate}
For every $r\in [0,1)$ there is a constant $C>0$ such that every bounded 
harmonic map $f:\D\to \R^n$ satisfies 
\begin{equation}\label{eq:estnabla}
	|\nabla f(z)| \le C \sup_{\zeta\in \D}|f(\zeta)-f(0)|,\quad\ |z|\le r.  
\end{equation}
\end{lemma}

%
%
%
\begin{proof}[Proof of Theorem \ref{th:completehyperbolic}]
Intersecting $\Omega$ with a small ball centred at $\bp$ and smoothing the corners
we obtain a bounded strongly minimally convex domain $\Omega_1\subset \Omega$ such that 
$\bp$ is contained in the relative interior of $b\Omega\cap b\Omega_1$.
Hence, in view of the localization (see Theorem \ref{th:loc1}), we may assume that $\Omega$ 
is bounded strongly minimally convex. Let $u$ be a strongly MPSH function on a neighbourhood of
$\overline \Omega$ such that $\Omega=\{u<0\}$ and $du\ne 0$ on $b\Omega=\{u=0\}$.
Note that for $\bx\in \Omega$, $-u(\bx)=|u(\bx)|$ is comparable to $\dist (\bx,b\Omega)$. 
Lemmas \ref{lem:distanceestimate} and \ref{lem:Cauchyestimate} 
provide a constant $c>0$ such that every $f\in\CH(\D,\Omega)$ whose
centre $f(0)$ is close enough to $b\Omega$ satisfies
\[
	|\nabla f(z)| \le c \sqrt{|u(f(0))|},\quad |z|\le \frac12.
\]
Together with \eqref{eq:mainformula} this gives
\begin{equation}\label{eq:estLaplace}
	|\Delta (u\circ f)(z)| \le c_1 |\nabla f(z)|^2\le C_1 |u(f(0))|, \quad |z|\le \frac12
\end{equation}
for some constant $c_1>0$ and $C_1=c_1c^2>0$. We claim that this implies 
\begin{equation}\label{eq:estgradient}
	|\nabla (u\circ f)(0)| \le C_2 |u(f(0))|,\quad f\in\CH(\D,\Omega) 
\end{equation}
with another constant $C_2>0$. 
By \cite[Lemma 1.1, p.\ 2396]{IvashkovichRosay2004}, the estimate \eqref{eq:estgradient} 
implies complete hyperbolicity of $\Omega$, so this will prove the theorem.

The proof of \eqref{eq:estgradient} is essentially the same as the proof 
in \cite{IvashkovichRosay2004} that (2.6) implies (2.4) in that paper. We recall the argument. 
Rescaling the disc $\D$ to $\frac12\D$, we may assume that 
\eqref{eq:estLaplace} holds for all $z\in \D$ (with a different constant). 
Set $v=u \circ f:\D\to (-\infty,0)$, so \eqref{eq:estLaplace} says that 
\begin{equation}\label{eq:estLaplace-v}	
	|\Delta v(z)|\le C_1|v(0)|=-C_1 v(0),\quad\ z\in\D.
\end{equation}
We extend $\Delta v$ to $\C$ by setting it equal to $0$ on $\C\setminus\overline\D$.
Consider the function
\begin{equation}\label{eq:gv}
	g(z)=v(z)- \Big( \frac{1}{2\pi} \log|\,\cdotp| * \Delta v\Big) (z)  - C_1 |v(0)|, \quad z\in \D,
\end{equation}
where $*$ denotes the convolution. Note that $\frac{1}{2\pi} \log|\,\cdotp| $ is the Green function 
for $0\in \D$, so $g$ is harmonic on $\D$. From \eqref{eq:estLaplace-v} it follows that  
\[
	\left| \Big(\frac{1}{2\pi} \log|\,\cdotp| * \Delta v \Big) (z)\right| \le C_1|v(0)|,
	\quad\ z\in\D.
\]
Hence, $g\le v<0$ on $\D$ and $|g(0)|<(2C_1+1)|v(0)|$. The Schwarz lemma for negative 
harmonic functions on $\D$ gives $|\nabla g(0)| \le 2|g(0)|$, and hence
\eqref{eq:estLaplace-v} and \eqref{eq:gv} imply
\[
	|\nabla v(0)| \le |\nabla g(0)| + \sup_\D |\Delta v| \le  2|g(0)| + C_1|v(0)| \le (5C_1+2) |v(0)|.
\]
This is the estimate \eqref{eq:estgradient} with the constant $C_2=5C_1+2$.

From \eqref{eq:estgradient}, completeness of $g_\Omega$ is seen by 
\cite[Lemma 1.1]{IvashkovichRosay2004}, replacing $(\chi,u)$ in their notation by the 
pair $(u,f)$ in our notation. (See also the proof of Lemma \ref{lem:halfspace}.)
\end{proof}

%
%
%
%
\section{Complete hyperbolicity of some unbounded strongly minimally convex domains}\label{sec:unbounded}
Combining Theorem \ref{th:completehyperbolic} with the localization theorem 
(see Theorem \ref{th:loc1}), we now give examples of unbounded strongly minimally convex 
domains $\Omega$ in $\R^n$ which are complete hyperbolic. The mentioned results show that 
any divergent path in $\Omega$ clustering at a finite boundary point of $\Omega$ has infinite length. 
It remains to ensure that the minimal length of any path in $\Omega$ diverging to infinity is infinite. 
Let us record this observation.

\begin{proposition}
If $\Omega\subset \Omega'\subset\R^n$ are (not necessarily bounded) 
domains such that $\Omega'$ is complete hyperbolic and $\Omega$ is strongly 
minimally convex, then $\Omega$ is also complete hyperbolic.
\end{proposition}

Estimating the length of a path diverging to infinity can sometimes be done with the help 
Theorem \ref{th:mpsh} (b). Here is a result in this direction.

%
%
\begin{theorem}\label{th:unbounded}
Let $\Omega\subset \R^n$, $n\ge 3$, be an unbounded strongly minimally convex domain.
Assume that there is a $\Cscr^2$ function $u:\Omega\to (-\infty,0)$ having
the following two properties:
\begin{enumerate}[\rm (a)]
\item There is $c>0$ such that $\tr_\Lambda \Hess_u(\bx)\ge c>0$ for all
$\bx\in \Omega$ and $\Lambda\in G_2(\R^n)$.
\item There is $c'>0$ such that $|u(\bx)| \le c'|\bx|^2$ for 
all $\bx\in \Omega$ with $|\bx|\ge 1$.
\end{enumerate}
Then $\Omega$ is complete hyperbolic.
\end{theorem}

\begin{proof}
As said above, any divergent path in $\Omega$ which terminates at a boundary point of $\Omega$
has infinite length. Suppose now that $\gamma:[0,1)\to\Omega$ is a path with 
$\lim_{t\to 1}|\gamma(t)|=\infty$. By condition (a) in the theorem and 
\eqref{eq:estg} in Theorem \ref{th:mpsh} there is a $C>0$ such that 
\[
	g_\Omega(\bx, \bv) \ge C \frac{|\bv|}{\sqrt{|u(\bx)|}} 
	\ \ \text{for all $\bx\in\Omega$ and $\bv\in\R^n$}. 
\]
In view of condition (b) we further get 
\[
	g_\Omega(\bx, \bv) \ge C' \frac{|\bv|}{|\bx|}
	\ \ \text{for all $\bx\in \Omega$ with $|\bx|\ge 1$ and $\bv\in\R^n$}
\]
with the constant $C'=C/\sqrt{c'}>0$. Hence,
\[
	g_\Omega(\gamma(t), \dot\gamma(t)) \ge C' \frac{|\dot\gamma(t)|}{|\gamma(t)|}
\]
holds for all $t$ close to $1$. It clearly follows that 
$\int_0^1 g_\Omega(\gamma(t), \dot\gamma(t)) dt=+\infty$.
\end{proof}

\begin{corollary}\label{cor:Omegaab}
For any $a<1$ and $b\in\R$ the domain 
\[
	\Omega_{a,b}=\left\{(x,y,z)\in\R^3: x^2 + y^2 < az^2 + b\right\}
\]
is complete hyperbolic. Hence, every strongly minimally convex domain contained 
in $\Omega_{a,b}$ is complete hyperbolic.  
\end{corollary}

\begin{proof} When $a<1$, the function $u(x,y,z)=x^2 + y^2 - az^2-b$ satisfies condition (a)
in Theorem \ref{th:unbounded} (see Example \ref{ex:Omegaab}), and it clearly satisfies 
$|u(x,y,z)|\le a|z|^2+b$ on $\Omega_{a,b}$. Hence, condition (b) in 
Theorem \ref{th:unbounded} holds as well, and the conclusion follows.
\end{proof}

%
%
%
%
\section{Extending conformal minimal surfaces across punctures}\label{sec:extend}
Recall that $\D^*=\D\setminus \{0\}=\{z\in \C:0<|z|<1\}$. We prove the following result.

%
%
\begin{theorem}\label{th:disc}
Let $\Omega$ be a hyperbolic domain in $\R^n$, $n\ge 3$, and let
$f:\D^*\to \Omega$ be a conformal harmonic map.  
If there is a sequence $z_k\in\D^*$ converging to $0\in \D$ such
that $f(z_k)$ converges to a point in $\Omega$, then $f$ extends
to a conformal harmonic map $\D\to\Omega$.  
\end{theorem}

The analogue of Theorem \ref{th:disc} for punctured holomorphic discs in Ko\-ba\-ya\-shi hyperbolic manifolds
(and, more generally, in hyperbolic complex spaces) is due to Kwack \cite{Kwack1969};
see also Kobayashi \cite[Theorem 3.1]{Kobayashi2005}. 
Kwack's theorem generalizes the big Picard theorem, which says that
a holomorphic map $\D^*\to\C\setminus\{0,1\}$ extends to a holomorphic
map $\D\to\CP^1=\C\cup\{\infty\}$. More generally, if $Y$ is a compact hyperbolic manifold,
$X$ is a connected complex manifold, and $A$ is a proper closed complex subvariety of $X$,
then every holomorphic map $X\setminus A\to Y$ extends to a holomorphic map $X\to Y$
\cite[Theorem 4.1]{Kobayashi2005}.

Theorem \ref{th:disc} fails in general if $\Omega$ is not hyperbolic.
For example, if $h$ is a holomorphic function on $\D^*$ with an essential singularity at $0$,
then $(h,0):\D^*\to\C\times\R=\R^3$ is a conformal harmonic map for which the theorem fails.
The same is true for the map 
$
	f=(\E^h,0):\D^*\to \Omega=(\C\setminus \{0\})\times \R\subsetneq\R^3.
$ 
(Note that $\Omega$ is the complement of a line in $\R^3$.)

Theorem \ref{th:disc} is trivial if $\Omega$ is a bounded domain  (since a bounded harmonic function 
extends harmonically across a puncture), but is nontrivial on unbounded domains. 
(Every bounded domain in $\R^n$ is hyperbolic by Corollary \ref{cor:bounded}.)
Combining with Theorem \ref{th:convex} gives:

%
%
\begin{corollary}
The conclusion of Theorem \ref{th:disc} holds in every convex domain $\Omega\subset \R^n$ 
$(n\ge 3)$ which does not contain any affine $2$-plane. 
\end{corollary}

%
%
\begin{proof}[Proof of Theorem \ref{th:disc}]
We begin with preliminaries. On the disc $\D=\{z\in\C:|z|<1\}$ we have the 
Poincar\'e--Bergman metric 
$
	\Pcal_\D=\frac{|dz|}{1-|z|^2}
$
of constant Gaussian curvature $-4$. 
(This normalization is used in the definition of the minimal metric $g_\Omega$.) 
The Poincar\'e metric on the punctured disc $\D^*=\{0<|z|<1\}$ of 
constant curvature $-4$, which makes the universal holomorphic covering map 
$\D\to \D^*$ a local isometry, equals
\[
	\Pcal_{\D^*}=\frac{|dz|}{2|z|\log(1/|z|)},\quad\ 0<|z|<1.
\]
(See \cite[p.\ 78]{Kobayashi2005}. Kobayashi uses a metric with constant 
curvature $-1$, so the multiplicative constants differ, which is irrelevant for the proof.)
The metric $\Pcal_{\D^*}$ is complete, and the distance from any fixed point $z_0\in\D^*$ to $z\in\D^*$ 
grows as $\frac12 \log\log (1/|z|) +O(1)$ when $z\to 0$ or $|z|\to 1$. 
The circle $|z|=r\in (0,1)$ has Poincar\'e length 
\begin{equation}\label{eq:Lr}
	L(r)= \frac{\pi}{\log 1/r}. 
\end{equation}
In particular, $L(r)$ decreases to $0$ as $r\to 0$. 
 
The scheme of our proof follows \cite[proof of Theorem 3.1]{Kobayashi2005}, but is considerably simpler. It also gives a simple proof of Kwack's theorem for punctured holomorphic discs in Kobayashi hyperbolic domains in $\C^n$.

Let $\bx=(x_1,\ldots,x_n)$ denote Euclidean coordinates on $\R^n$.
Assume that $f$ satisfies the hypotheses of Theorem \ref{th:disc}, so we have
a sequence $z_k\in\D^*$ with
\begin{equation}\label{eq:fzk}
	 \lim_{k\to\infty}z_k=0\ \ \text{and}\ \  \lim_{k\to\infty} f(z_k)=\bp\in \Omega.
\end{equation} 
We must must show that for every $\epsilon>0$ there is a $\delta\in (0,1)$ such that $f$ 
maps the punctured disc $\delta \D^*= \{z\in\C:0<|z|<\delta\}$ into the ball 
$\B(\bp,\epsilon)$. We shall assume that $\epsilon>0$ is small enough such that 
$\overline \B(\bp,\epsilon)\subset \Omega$.
 
Passing to a subsequence, we may assume that the sequence
$r_k=|z_k|>0$ $(k\in\N)$ is strictly monotonically decreasing to $0$. 
Let $\gamma_k=\{z\in\C:|z|=r_k\}$, $k\in\N$. 
By \eqref{eq:Lr} we have that $\lim_{k\to\infty} L(r_k)=0$. 
Since $f:\D^*\to\Omega$ is metric-decreasing from the
Poincar\'e metric on $\D^*$ to the minimal metric $g_\Omega$ on $\Omega$, 
the length of the curve $f(\gamma_k)\subset \Omega$ with respect to $g_\Omega$ 
converges to zero as $k\to\infty$. 
Since $\Omega$ is hyperbolic, $g_\Omega$ induces the standard topology 
on $\Omega$ by Theorem \ref{th:hyperbolic}. 
Hence, since $z_k\in\gamma_k$ and $f(z_k)\to \bp$ as $k\to\infty$, 
there is a $k_0\in\N$ such that
$
	f(\gamma_k)\subset \B(\bp,\epsilon)\ \ \text{for all}\ \ k\ge k_0.
$ 
Let 
\[
	A_k=\{z\in \C: r_{k+1}<|z|<r_k\} \subset \D^*,\quad \  k\in\N.
\]
Note that $bA_k=\gamma_k\cup \gamma_{k+1}$, and hence $f(bA_k)\subset \B(\bp,\epsilon)$
for $k\ge k_0$. By the maximum principle for minimal surfaces it follows that 
$f(A_k)\subset \B(\bp,\epsilon)$ for all $k\ge k_0$. Since 
\[
	\bigcup_{k=k_0}^\infty A_k =\{z\in\D: 0<|z|\le r_{k_0}\} = r_{k_0}\D^*,
\]
we see that $f( r_{k_0}\D^*)\subset  \B(\bp,\epsilon)$. 
\end{proof}

%
%

We now give an application of Theorem \ref{th:disc}.
By analogy with domains in $\C^n$, we introduce the following notion
(see Kobayashi \cite[p.\ 190]{Kobayashi1998}). 

\begin{definition}\label{def:hyperconvex}
A domain $\Omega\subset\R^n$ $(n\ge 3)$  is {\em hyperconvex} if it admits
a continuous negative MPSH exhaustion function $u:\Omega\to [-\infty,0)$.
\end{definition}

Theorem \ref{th:mpsh} (d) says that every hyperconvex domain is hyperbolic.
Such a domain may be unbounded as shown by Example \ref{ex:hyperconvex}.

%
%
\begin{corollary}\label{cor:hyperconvex}
If $\Omega \subset\R^n$ $(n\ge 3)$ is a hyperconvex domain, then
any conformal harmonic map $\D^*\to \Omega$ extends
to a conformal harmonic map $\D\to \Omega$.

In particular, $\Omega$ does not contain any nonconstant conformal minimal surface 
$M\to \Omega$ such that $M$ is the complement of finitely many points in a 
compact Riemann surface.
\end{corollary}

It is easy to find a bounded (hence hyperbolic) domain $\Omega\subset \R^3$
and a conformal harmonic punctured disc $f:\D^*\to \Omega$, which extends to a 
conformal harmonic disc $f:\D\to \overline\Omega$ but $f(0)\in b\Omega$. 
(See Example \ref{ex:notcompletehyperbolic}.) Hence, Corollary \ref{cor:hyperconvex} 
fails in general on hyperbolic domains which are not hyperconvex.

\begin{proof}[Proof of Corollary \ref{cor:hyperconvex}.]
Let $u:\Omega\to (-\infty,0)$  be a continuous negative MPSH exhaustion function.
Assume that $f:\D^*\to \Omega$ is a conformal harmonic map.
If $f$ does not extend to the origin $0\in \D$ with $f(0)\in \Omega$, 
then Theorem \ref{th:disc} shows that for every sequence $z_k\in \D$
with $\lim_{k\to\infty} z_k=0$ the sequence $f(z_k)$ diverges to $b\Omega\cup \{\infty\}$.
Since $u$ is a negative exhaustion function on $\Omega$, it follows that
$\lim_{k\to\infty} u(f(z_k))=0$. This shows that $u\circ f:  \D^* \to (-\infty,0)$ extends to 
a continuous function $h:\D\to (-\infty,0]$ with $h(0)=0$. Since $f$ is conformal harmonic 
on $\D^*$ and $u$ is MPSH, $u\circ f$ is subharmonic on $\D^*$, and hence the extension
$h$ is subharmonic on $\D$. The fact that the extended function reaches 
maximal value $h(0)=0$ at $0\in\D$  contradicts the maximum principle for 
subharmonic functions. 
This contradiction shows that $f$ extends across the origin to a conformal harmonic map 
$\D\to\Omega$. The last statement follows from the fact that every harmonic function
on a compact Riemann surface $\wh M$ is constant.
\end{proof}

%
%
%
%
\section{Open problems}\label{sec:problems}

\noindent{\bf A. Continuity of the minimal metric.}
Let $\Omega$ be a domain in $\R^n$, $n\ge 3$. Recall from Section \ref{sec:definition} 
that the minimal pseudometric $g_\Omega:\Omega\times \R^n\to\R_+$ is upper-semicontinuous.

\begin{problem}
Is there a domain $\Omega\subset\R^n$ whose minimal metric $g_\Omega$
is not continuous?
\end{problem}

There are domains $\Omega\subset \C^2$ with discontinuous Kobayashi metric.
We wish to thank Peter Pflug for the following example (private communication).
There is a pseudoconvex balanced domain $\Omega\subset \C^n$ for any $n>1$ 
with a non-continuous Minkowski function $h$. Then the Kobayashi metric of 
$\Omega$ at the origin equals $\Kcal_\Omega(0;\cdot)=h$, hence it is not continuous. 
See \cite[Theorem 3.5.3]{JarnickiPflug2013} for the details.

%
%
\smallskip
\noindent{\bf B. Relationship between $\Mcal_\Omega$ and $g_\Omega$.}
Recall (see \eqref{eq:gM}) that for any unit vector $\bv\in\R^n,\ |\bv|=1$ we have
$ 
	g_\Omega(\bx,\bv) = 
	\inf\big\{ \Mcal_\Omega(\bx,\Lambda): \Lambda\in  G_2(\R^n), \bv\in\Lambda\}. 
$
On the ball $\B^n\subset\R^n$ we also have that 
\begin{equation}\label{eq:max}
	\Mcal_{\B^n}(\bx,\Lambda) = \max\bigl\{g_{\B^n}(\bx,\bv) : \bv\in\Lambda, |\bv|=1\bigr\}.
\end{equation}
(See \cite[Eq.\ (2.10)]{ForstnericKalaj2021}.) 
This fails in general as the following example shows. 

%
%
\begin{example} \label{ex:nonregular} 
Let $\Omega=\B^2\times\R\subset\R^3$. Taking $\Lambda=\R^2\times \{0\}$ we have that 
$\Mcal_\Omega(\zero,\Lambda)=1$, the extremal disc being the linear embedding
$x+\imath y \mapsto (x,y,0)$. Indeed, given $f=(f_1,f_2,f_3)\in \CH(\D,\Omega)$
with $f(0)=\zero$ and $df_0(\R^2)=\Lambda$, the projection $(f_1,f_2):\D\to \Lambda\cap \Omega=\D$
is a harmonic map which is conformal at the origin, so the claim follows from the Schwarz--Pick
lemma due to Forstneri\v c and Kalaj \cite[Theorem 1.1]{ForstnericKalaj2021}. 
Given a unit vector $\bu=(u_1,u_2,0)\in \Lambda$, take $\bv=(0,0,1)$ and let $\Sigma_\bu$ 
be the $2$-plane spanned by the orthonormal frame $(\bu,\bv)$. Then, 
$\Sigma_\bu \cap\Omega$ is conformally equivalent to the strip
$S=(-1,+1)\times \R\subset \C$. The function
\[
	f(z)=\frac{2\imath}{\pi}\log\frac{1+z}{1-z}, \quad \ z\in\D
\]
provides a conformal diffeomorphism of $\D$ onto $S$. We have 
$f'(0)=4\imath/\pi$, so $\Mcal_\Omega(\zero,\Sigma_\bu)\le 1/\|df_0\| = \pi/4$.
Hence, $g_\Omega(\zero,\bu)\le \pi/4 <1=\Mcal_\Omega(\zero,\Lambda)$ 
for all unit vectors $\bu\in \Lambda$. 
\qed\end{example}

The argument in the above example shows that property \eqref{eq:max} fails on any 
bounded strongly convex domains in $\B^2\times\R$ with the base $\B^2\times\{0\}$ 
which is sufficiently big in the third coordinate direction. 
Hence, it seems that the ball is rather exceptional in this sense.

%
%
\begin{problem}
Let $\Omega$ be a bounded domain in $\R^n$ $(n\ge 3)$ such that 
\[
	\Mcal_\Omega(\bx,\Lambda) = \max\bigl\{g_\Omega(\bx,\bv) : \bv\in\Lambda, |\bv|=1\bigr\}
\]
holds for every $\bx\in\Omega$ and $\Lambda\in  G_2(\R^n)$. Is $\Omega$ is
equivalent to $\B^n$ by a rigid transformation?
\end{problem}

%
%
\smallskip
\noindent{\bf C. Distance to a minimal surface.}
By Lemma \ref{lem:halfspace}, any point in a half-space $\H\subset\R^n$
is at infinite minimal distance from the hyperplane $\Sigma=b\H$. 
This fails for more general minimal hypersurfaces if $n\ge 4$; however,
the following is an interesting question in dimension $3$. 

\begin{problem}\label{pr:distancetohypersurface}
Let $\Sigma$ be an embedded minimal surface in a bounded domain 
$\Omega \subset\R^3$. Is the minimal distance in $\Omega\setminus \Sigma$ 
to any point $\bp\in \Sigma$ infinite? 
\end{problem}

The analogous problem for the Kobayashi metric  has affirmative answer:
the distance to a complex hypersurface is infinite, also in the nonintegrable case; see
Ivashkovich and Rosay \cite{IvashkovichRosay2004}. Hence, if a bounded weakly 
pseudoconvex domain $D$ in $\C^n$ contains a Levi-flat piece in the boundary $bD$,
then the Kobayashi distance to it is also infinite. Concerning more general weakly pseudoconvex domains,
Catlin proved in 1989 \cite{Catlin1989} that every bounded weakly pseudoconvex domain 
of finite type in $\C^2$ is complete hyperbolic. A partial result in the nonintegrable case 
was obtained by Bertrand \cite{Bertrand2012} in 2012.
It would be of interest to see whether minimally convex domains behave better in this respect.

\begin{problem}
Suppose that $\Omega\subset\R^n$ is a bounded minimally convex domain with smooth
boundary; is $\Omega$ necessarily complete hyperbolic in the minimal metric $g_\Omega$?
\qed\end{problem}

%
%
\smallskip
\noindent{\bf D. Extremal minimal discs in convex domains.}
If $\Omega$ is a bounded strongly convex domain in $\C^n$, then by the seminal result
of Lempert \cite{Lempert1981,Lempert1982} there exists for every point 
$\bp\in \Omega$ and vector $0\ne \xi \in \C^n$ a unique extremal holomorphic disc 
$F=F_{p,\xi}:\D\to \Omega$ such that 
$F(0)=\bp$ and the derivative $F'(0)=r\xi$ $(r>0)$ is the biggest possible. 
This extremal disc is properly embedded in $\Omega$, and it is 
an isometry from $\D$ with the Poincar\'e metric 
onto its image $F(\D)\subset \Omega$
with the Kobayashi metric $\Kcal_\Omega$. Extremal holomorphic discs also exist in taut 
unbounded convex domains of $\C^n$ (see \cite[Lemma 3.3]{BracciSaracco2009}).

\begin{problem}\label{prob:Lempert}
Does the analogue of Lempert's theorem hold for conformal harmonic discs in a (smoothly) 
bounded convex domain $\Omega\subset\R^n$, $n\ge 3$? 
\end{problem}

The answer is affirmative if $\Omega$ is the ball $\B^n$ of $\R^n$, in which case the extremal discs 
are the proper affine discs in $\B^n$ (see Example \ref{ex:ball} and \cite[Theorem 2.1]{ForstnericKalaj2021}).
It seems that the ball is the only domain for which the answer to Problem \ref{prob:Lempert} 
is known at this time.

%
%
\smallskip
\noindent{\bf E. Gromov hyperbolicity of strongly minimally convex domains.}
It was shown by Balogh and Bonk \cite{BaloghBonk2000} that the Kobayashi metric on 
any smooth bounded strongly pseudoconvex domain in $\C^n$ is Gromov hyperbolic.
Since in many respects strongly pseudoconvex domains play a similar role in complex analysis as the 
strongly minimally convex domains play in the theory of minimal surfaces, the following is a natural
question. 

\begin{problem}\label{prob:Gromovhyperbolic}
Is every smooth bounded strongly minimally convex domain in $\R^n$ Gromov hyperbolic
in the minimal metric? 
\end{problem}

%
%




\end{document}